\documentclass{article}

\RequirePackage{amsthm,amsmath,amsfonts,amssymb}
\RequirePackage[numbers]{natbib}

\usepackage[utf8]{inputenc}
\usepackage{algorithm}
\usepackage{algorithmic}
\usepackage{blkarray}
\usepackage{comment}
\usepackage{mathtools}
\usepackage{stmaryrd}

\usepackage{MnSymbol}

\theoremstyle{plain}
\newtheorem{thm}{Theorem}

\newtheorem{cor}[thm]{Corollary}
\newtheorem{lem}[thm]{Lemma}
\newtheorem{prop}[thm]{Proposition}
\theoremstyle{remark}
\newtheorem{defn}[thm]{Definition}
\newtheorem{exmp}[thm]{Example}
\newtheorem{rem}[thm]{Remark}

\usepackage{tikz}
\usepackage{tikz-cd} 
\usepackage{tkz-graph} 
\usetikzlibrary{matrix,arrows, arrows.meta,automata, calc, 
	decorations.pathmorphing,shapes, backgrounds, positioning,
	shapes.geometric}

\newcommand{\D}{\mathcal{D}}

\newcommand{\G}{\mathcal{G}}
\newcommand{\I}{\mathcal{I}}

\newcommand{\starrightarrow}{\ *\!\!\rightarrow}
\newcommand{\starleftarrow}{\leftarrow\!\! * \ }

\newcommand{\msep}[3]{#1 \perp_m #2 \mid #3}
\newcommand{\notmsep}[3]{#1 \not\perp_m #2 \mid #3}
\newcommand{\msepG}[4]{#1 \perp_m #2 \mid #3\ [#4]}
\newcommand{\musep}[3]{#1 \perp_\mu #2 \mid #3}
\newcommand{\musepG}[4]{#1 \perp_\mu #2 \mid #3\ [#4]}

\newcommand{\an}{\text{an}}
\newcommand{\de}{\text{de}}
\newcommand{\md}{\ \mathrm{d}}
\newcommand{\pa}{\text{pa}}

\newcommand{\ito}{It\^o\ }

\newcommand{\disjU}{\mathbin{\dot{\cup}}}

\newcommand{\E}[2]{\mathrm{E}\!\left[#1 \mid #2\right]}

\DeclareRobustCommand{\unEdge}{%
	\mathrel{%
		\text{%
			\ooalign{$\rightfootline$\cr\reflectbox{$\rightfootline$}\cr}%
		}%
	}%
}

\tikzset{%
	myCircle/.style={%
		draw = black,shape = circle, minimum size = {height("$\beta$") + 14pt}
	}
}
\tikzset{%
	myRect/.style={%
		draw = black,shape = rectangle, minimum size = {height("$\beta$") + 
			12pt}
	}
}
\tikzset{
	every circle node/.style=myCircle,
	every edge/.append style = {shorten >= 2pt, shorten <= 2pt}
}

\title{Graphical modeling of stochastic processes driven by correlated 
	errors}
\author{Søren Wengel Mogensen$^{1,2,}$\footnote{swemo@dtu.dk} \and Niels 
Richard Hansen$^1,$\footnote{niels.r.hansen@math.ku.dk}}
\date{\small \it $^1$ Department of Mathematical Sciences, University of 
	Copenhagen \\
	$^2$ Section for Cognitive Systems, 
	Technical University of 
	Denmark}

\begin{document}

\maketitle

	\begin{abstract}
			We study a class of graphs that represent local independence 
			structures 
			in stochastic processes allowing for correlated 
			error processes.  Several graphs may encode the same local 
			independencies 
			and we characterize such equivalence classes of graphs. In the 
			worst case, the number 
			of 
			conditions in our characterizations grows superpolynomially as a 
			function 
			of 
			the size of the node set in the graph. We show that deciding Markov 
			equivalence is 
			coNP-complete which suggests that our characterizations cannot be 
			improved 
			upon substantially. We prove 
			a global 
			Markov property in the case of a 
			multivariate Ornstein-Uhlenbeck process which is driven by 
			correlated 
			Brownian motions.
	\end{abstract}


	
\section{Introduction}

Graphical modeling studies how to relate graphs to properties of probability 
distributions \citep{lauritzen1996}. There is a rich literature on graphical 
modeling of distributions of multivariate random variables 
\citep{maathuis2018}, in particular on graphs as representations of conditional 
independencies. In stochastic processes, local independence can be used as a 
concept analogous to conditional 
independence and several papers use graphs to encode local independencies 
\citep{didelez2006, didelez2008, aalen2012, roysland2012, mogensenUAI2018, 
Mogensen2020a}.
\citet{didelez2000, didelez2008} studies graphical modeling of local 
independence 
of multivariate point processes. 
\citet{mogensenUAI2018} also consider diffusions. This previous work only 
models 
direct influence between coordinate processes in a multivariate stochastic 
process. We consider the more general case in which the error processes driving 
the continuous-time stochastic process may be correlated. \citet{eichler2007, 
eichlerGranger2007, 
	eichlerChapter2012, eichler2013} study this in the time series case (i.e., 
	stochastic 
	processes indexed by discrete time). 
	
	A specific local 
independence 
structure can be represented by several different graphs, and the 
characterization of such Markov equivalence classes is an important 
question in graphical modeling. We study these equivalence classes and 
characterize them. Our characterizations are computationally demanding as they 
may 
involve exponentially many conditions (as a function of the number of nodes in 
the graphs). We prove that deciding Markov 
equivalence in this class of graphs is coNP-hard, and therefore one would not 
except to find a characterization which is verified in polynomial time.

Markov properties are central in graphical modeling as they allow us to deduce 
independence from graphs. The graphical results in this paper apply to various 
classes of stochastic 
processes for which it is possible to show a so-called global Markov property. 
As an example, we study systems of linear stochastic 
differential equations (SDEs), 
and in particular 
Ornstein-Uhlenbeck processes. Such models have been used in 
numerous fields such as psychology \citep{heath2000}, neuroscience 
\citep{ricciardi1979, shimokawa2000, ditlevsen2005}, finance \citep{stein1991, 
	schobel1999, 
	bormetti2010}, 
	biology \citep{bartoszek2017}, and survival analysis \citep{aalen2004, 
	lee2006}. In this paper, we show that Ornstein-Uhlenbeck processes with 
	correlated driving 
Brownian motions satisfy a global Markov property with respect to a certain 
graph. Previous work in continous-time models considers independent error 
processes only and the present work extends this framework to cases where the 
driving processes are correlated. To our knowledge, our result is the first 
such result in continuous-time models. It is analogous to results in time 
series models with correlated error processes \cite{eichler2007, 
eichlerGranger2007, 
	eichlerChapter2012, eichler2013}. The graphical and algorithmic results 
we present also 
apply to these time series models.

The paper is organized as follows. Section \ref{sec:LI} introduces local 
independence for \ito processes.  Section
\ref{sec:graphLI} defines  {\it directed correlation graphs} (cDGs) -- the
class of graphs that we will use throughout the paper to represent  local
independencies in a stochastic process. In Section  \ref{sec:graphLI} we state
a global Markov property for Ornstein-Uhlenbeck processes. Section  \ref{sec:ME} gives
a   characterization of the cDGs that encode the same independencies. This
directly leads to  an algorithm for checking equivalence of cDGs. This algorithm
runs in  exponential time (in the number of nodes in the graphs). In Section 
\ref{sec:decME} we state another characterization of Markov equivalence and we 
prove that deciding Markov equivalence is
coNP-complete.


\section{Local independence}
\label{sec:LI}


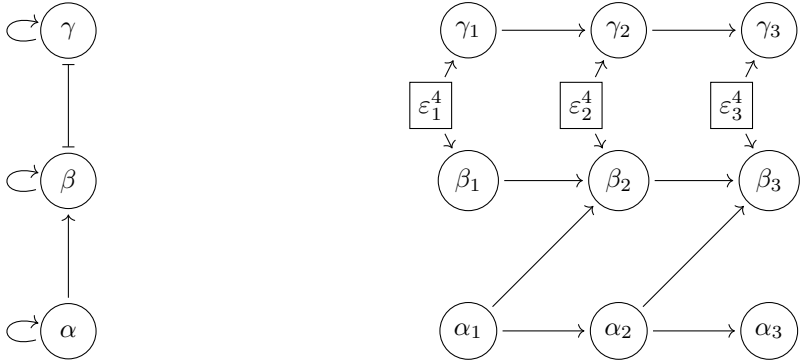
\begin{figure}
	\begin{minipage}{0.45\linewidth}
		\begin{tikzpicture}
		\node[shape=circle,draw=black] (a) at (0,0) {$\alpha$};
		\node[shape=circle,draw=black] (b) at (0,2) {$\beta$};
		\node[shape=circle,draw=black] (c) at (0,4) {$\gamma$};
		
		\path [->] (a) edge [bend left = 0] node {} (b);
		\path [|-|](b) edge [bend right = 0] node {} (c);
		\draw[every loop/.append style={->}] (c) edge[loop left] 
		node {}  (c);
		\draw[every loop/.append style={->}] (b) edge[loop left] 
		node {}  (b);
		\draw[every loop/.append style={->}] (a) edge[loop left] 
		node {}  (a);	
		\end{tikzpicture}
	\end{minipage}
	\begin{minipage}{0.45\linewidth}
		\begin{tikzpicture}
		\node[shape=circle,draw=black] (a1) at (0,0) {$\alpha_{1}$};
		\node[shape=circle,draw=black] (b1) at (0,2) {$\beta_{1}$};
		\node[shape=circle,draw=black] (c1) at (0,4) {$\gamma_{1}$};
		\node[shape=circle,draw=black] (a2) at (2,0) {$\alpha_{2}$};
		\node[shape=circle,draw=black] (b2) at (2,2) {$\beta_{2}$};
		\node[shape=circle,draw=black] (c2) at (2,4) {$\gamma_{2}$};
		\node[shape=circle,draw=black] (a3) at (4,0) {$\alpha_{3}$};
		\node[shape=circle,draw=black] (b3) at (4,2) {$\beta_{3}$};
		\node[shape=circle,draw=black] (c3) at (4,4) 
		{$\gamma_{3}$};
		\node[shape=rectangle,draw=black] (d1) at (-0.5,3) {$\varepsilon_{1}^4$};
		\node[shape=rectangle,draw=black] (d2) at (1.5,3) {$\varepsilon_{2}^4$};
		\node[shape=rectangle,draw=black] (d3) at (3.5,3) 
		{$\varepsilon_{3}^4$};		
		
		\path [->] (a1) edge [bend left = 0] node {} (b2);
		\path [->] (a1) edge [bend left = 0] node {} (a2);
		\path [->] (b1) edge [bend left = 0] node {} (b2);
		\path [->] (c1) edge [bend left = 0] node {} (c2);
		
		\path [->] (a2) edge [bend left = 0] node {} (b3);
		\path [->] (a2) edge [bend left = 0] node {} (a3);
		\path [->] (b2) edge [bend left = 0] node {} (b3);
		\path [->] (c2) edge [bend left = 0] node {} (c3);
		
		
		\path [->] (d1) edge [bend left = 0] node {} (b1);
		\path [->] (d1) edge [bend left = 0] node {} (c1);
		\path [->] (d2) edge [bend left = 0] node {} (b2);
		\path [->] (d2) edge [bend left = 0] node {} (c2);
		\path [->] (d3) edge [bend left = 0] node {} (b3);
		\path [->] (d3) edge [bend left = 0] node {} (c3);
		\end{tikzpicture}
		
	\end{minipage}
	\caption{A local independence graph (left) and an `unrolled' graph (right) 
	where time  is
	made explicit. The two graphs  represent the same local independence
	structure. A node  $\delta$ for $\delta \in \{\alpha,\beta,\gamma\}$
	represents the increments of the $X_t^\delta$-process at  time $t$. On the
	right, the  $\varepsilon^4$-process is a `white noise' process that creates
	dependence  between $X^\beta_t$ and $X^\gamma_t$. In the `rolled' version of the
	graph  (left) this is represented by a {\it blunt} edge, 
	$\beta\unEdge\gamma$. When unrolling a local independence graph to 
	obtain a graphical representation in terms of lagged variables, we could 
	also 
	choose to include $\alpha_s \rightarrow \beta_t$ in the 
	unrolled graph for all $s<t$ if 
	$\alpha\rightarrow\beta$ in the local independence graph (see also 
	\cite{sokol2014, danks2013, Hyttinen2016} and
	\cite[supplementary material]{Mogensen2020a}).}
	\label{fig:unrolledcDG}
\end{figure}

Before diving into a formal introduction, we will consider a motivating example.

\begin{exmp}
Consider the three-dimensional Ornstein-Uhlenbeck process, which solves the 
following stochastic differential equation,
$$
\md \left(\begin{array}{c} X_t^\alpha \\ X_t^\beta \\ X_t^\gamma 
\end{array}\right) 
= \underbrace{\left( \begin{array}{ccc} M_{\alpha \alpha} & 0 & 0 \\
M_{\beta \alpha} & M_{\beta \beta} & 0 \\ 
0 & 0 & M_{\gamma \gamma} \end{array} \right)}_{=M} \left(\begin{array}{c} X_t^\alpha \\ X_t^\beta \\ X_t^\gamma \end{array}\right) \mathrm{d} t + 
\underbrace{\left( \begin{array}{cccc} \sigma_{\alpha} & 0 & 0 & 0 \\
0 & \sigma_{\beta} & 0 & \rho_{\beta} \\ 
0 & 0 & \sigma_{\gamma} & \rho_{\gamma} \end{array} \right)}_{= \sigma_0} \mathrm{d} 
\left(\begin{array}{c} W_t^1 \\ W_t^2 \\ W_t^3 \\ W_t^4 \end{array}\right)
$$
where $(W_t^1, W_t^2, W_t^3, W_t^4)^T$ is a standard four-dimensional 
Brownian motion. In this example, all entries in the matrix $M$ above that are not 
explicitly 0 are assumed nonzero and likewise for $\sigma_0$.

The interpretation of the stochastic differential equation via 
the Euler-Maruyama scheme yields the update equation
\begin{align*}
 \tilde{X}_{t+\Delta}^\alpha & = \tilde{X}_t^\alpha + \Delta M_{\alpha \alpha} 
 \tilde{X}_t^\alpha  + \sqrt{\Delta} \sigma_{\alpha} \varepsilon_t^{1} \\
 \tilde{X}_{t+ \Delta}^\beta  & = \tilde{X}_t^\beta + \Delta 
(M_{\beta \alpha} \tilde{X}_t^{\alpha}  + 
M_{\beta \beta} \tilde{X}_t^\beta) + \sqrt{\Delta} \left(\sigma_{\beta} 
\varepsilon_t^{2} +  
\rho_{\beta} \varepsilon_t^{4}\right) \\
\tilde{X}_{t+\Delta}^\gamma & = \tilde{X}_t^\gamma +
\Delta M_{\gamma \gamma} \tilde{X}_t^\gamma  + 
\sqrt{\Delta} \left( \sigma_{\gamma} \varepsilon_t^{3} + \rho_{\gamma} \varepsilon_t^{4}\right) 
\end{align*}
where $\varepsilon_t \sim \mathcal{N}(0, I)$.
%
%
The Euler-Maruyama scheme evaluated in $t = n \Delta$ for $n \in \mathbb{N}_0$
gives a process, $(\tilde{X}_{n\Delta})_{n\geq0}$,  which, as $\Delta \to 0$,
converges to the Ornstein-Uhlenbeck process, $(X_t)_{t \geq 0}$, solving the
stochastic differential equation. From the update equations we see  that the
infinitesimal increment of each coordinate depends on that coordinate's own
value, and coordinate $\beta$ depends, in addition, on coordinate $\alpha$
(because $M_{\beta \alpha} \neq 0$). Moreover, the increments for coordinates
$\beta$ and $\gamma$ are correlated as they share the error variable
$\varepsilon_t^4$. Figure \ref{fig:unrolledcDG} (left) provides a graphical
representation with arrows readily read off from the drift matrix $M$ and the
diffusion matrix $\sigma \sigma^T$. The `unrolled' graph (Figure 
\ref{fig:unrolledcDG}, right)
is a \emph{directed acyclic graph} (DAG) which corresponds to the 
Euler-Maruyama scheme and provides a discrete-time representation of the 
dynamics. 
\label{ex:simpleOU}
\end{exmp}

A central purpose of this paper is to clarify the mathematical
interpretation of {\it local independence graphs} such as the one in Figure 
\ref{fig:unrolledcDG} (left), 
and our results include a characterization of all graphs with equivalent
mathematical content. As showcased in the example above, we allow for a 
nondiagonal $\sigma_0 \sigma_0^T$ which is a novelty in graphical modeling of 
continuous-time stochastic processes.

\subsection{\ito processes and local independence graphs}

We will for the purpose of this paper focus on vector-valued, continuous-time
stochastic processes with continuous sample paths. Thus let $X = (X_t)_{t \in
\mathcal{T}}$ denote such an $n$-dimensional process with time index $t \in
\mathcal{T} \subseteq \mathbb{R}$ and with $X_t = (X_t^{\alpha})_{\alpha \in
[n]} \in \mathbb{R}^n$ being a real-valued vector indexed by $[n] = \{1, \ldots,
n\}$. The time index set $\mathcal{T}$ will in practice be of the forms $[0,T]$,
$[0,\infty)$, or $\mathbb{R}$, however, we  will in general just assume that
$\mathcal{T}$ is an interval containing  $0$.  

We use \emph{local  independence} \cite{schweder1970,aalen1987, 
	didelez2008,Commenges:2009} to give a mathematically precise
definition of what it means for the historical evolution of one coordinate,
$\alpha$,  to  \emph{not} be predictive of the infinitesimal increment of
another coordinate,  $\beta$, given the historical evolution of a set, $C
\subseteq [n]$, of  coordinates. As such, it is a continuous-time version of
Granger causality \citep[see, e.g.,][]{grangerForecast1986}, 
and its formulation is
directly related to filtration problems for stochastic processes. In a
statistical  context, local independence allows us to express simplifying
structural constraints that are directly useful for forecasting and
such constraints are also useful for causal structure learning. 

The process $X$ is defined on the probability space $(\Omega, \mathcal{F}, P)$
and we let $\sigma(X_s^{\delta}; s \leq t, \delta \in D) \subseteq \mathcal{F}$
denote the  $\sigma$-algebra on $\Omega$ generated by $X_s^{\delta}$ for all $s
\in \mathcal{T}$ up to time $t$ and all $\delta \in D$. For technical reasons,
we  define $\mathcal{F}_t^D$ to be the $P$-completion of the $\sigma$-algebra
$$\bigcap_{t' > t} \sigma(X_s^{\delta}; s \leq t', \delta \in D),$$ so that
$(\mathcal{F}_t^D)_{t \in \mathcal{T}}$ is a complete, right-continuous
filtration for all $D \subseteq [n]$. We will let $\mathcal{F}_t = \mathcal{F}_t^{[n]}$
denote the filtration generated by all coordinates of the process. 
Within this setup we will restrict attention to \ito processes with continuous 
drift and constant diffusion coefficient.  

\begin{defn}[Regular \ito processes] \label{def:regIto}  We say that $X$ is a
\emph{regular \ito process} if there exists a  continuous, 
$\mathcal{F}_t$-adapted
process, $\lambda$, with values in  $\mathbb{R}^n$, and an $n \times n$
invertible matrix $\sigma$ such that  $$W_t = \sigma^{-1} \left(X_t - X_0 -
\int_0^t \lambda_s \mathrm{d}s\right)$$ is an $\mathcal{F}_t$-adapted standard
Brownian motion.
\end{defn}

One reason for the 
interest in the general class of \ito processes is that they are closed 
under marginalization. A regular \ito process is sometimes written in 
differential form as 
\begin{equation} \label{eq:ItoSDE}
\md X_t = \lambda_t \mathrm{d}t + \sigma \md W_t.
\end{equation}
Here $\lambda_t$ is known as the drift of the process and $\sigma$ as the
(constant) diffusion coefficient. We define the \emph{diffusion matrix} for
a regular \ito process as the positive definite matrix
\begin{equation} \label{eq:diffmat}
\Sigma = \sigma\sigma^T.
\end{equation}
Observe that the process $X_t$ may, as in Example \ref{ex:simpleOU},  be defined
as the solution of the stochastic differential equation
\begin{equation} \label{eq:SDE}
\md X_t = \lambda_t \mathrm{d}t + \sigma_0 \md W_t
\end{equation}
for an $m$-dimensional standard Brownian motion $W$ and with the  diffusion
coefficient $\sigma_0$ an $n \times m$ matrix. If $\sigma_0$ has  rank $n$, such
a solution is also a regular \ito process with diffusion matrix $\Sigma =
\sigma_0 \sigma_0^T$. Indeed, we can take $\sigma = (\sigma_0 \sigma_0^T)^{1/2}$
in Definition \ref{def:regIto}. Observe also that for any regular \ito process,
$$X_t - X_0 -  \int_0^t \lambda_s \mathrm{d}s = \sigma W_t$$
is an $\mathcal{F}_t$-martingale and $\int_0^t \lambda_s \mathrm{d}s$ is 
the compensator of $X_t$ in its Doob-Meyer decomposition.


\begin{defn} Let $X$ be a regular \ito process with drift $\lambda$,
let $\alpha, \beta \in [n]$, and let $C \subseteq [n]$. 
We say that \emph{$\beta$ is locally independent of $\alpha$ given $C$}, 
and write $\alpha \not\rightarrow  \beta \mid C$, if the process 
$$ t \mapsto E\left(\lambda^{\beta}_t \mid \mathcal{F}_t^{C}\right)$$
is a version of $t \mapsto E\left(\lambda^{\beta}_t \mid \mathcal{F}_t^{C \cup 
\{\alpha\}}\right)$.
\label{def:locindSingletonA}
\end{defn} 

It follows immediately from the definition that  $\alpha \not\rightarrow  \beta
\mid [n] \setminus \{ \alpha\}$ if $\lambda_t^{\beta}$  is $\mathcal{F}_t^{[n]
\setminus \{\alpha \}}$-measurable. That is, if $\lambda_t^{\beta}$ does not depend on
the sample path of the $\alpha$-coordinate.


We define a {\it local independence graph} below and this generalizes the 
definitions
of  \citet{didelez2008} and \citet{Mogensen2020a} in the context of 
continuous-time  
stochastic
processes to  allow for a nondiagonal $\Sigma$. \citet{eichlerGranger2007}
gives a related definition in the case of time series (discrete time) with 
correlated
errors and uses the term {\it path diagram} (see also Definition 
\ref{def:CLIG}).

\begin{defn}[Local independence graph] 
Consider a regular \ito process with diffusion matrix $\Sigma$. A \emph{local
independence graph} is a graph, $\mathcal{D}$, with nodes $[n]$  such that for 
all $\alpha,\beta\in [n]$ 
$$\alpha \not \rightarrow_\D \beta \ \ \Rightarrow \ \ 
\alpha \not\rightarrow  \beta
\mid [n] \setminus \{ \alpha\}$$
 and such that for $\alpha \neq \beta$
$$\alpha \not \unEdge_{\D} \beta \ \ \Rightarrow  \ \
\Sigma_{\alpha \beta} = 0$$
where $\rightarrow_\D$ denotes a directed edge in $\D$ and $\alpha \unEdge_{\D} 
\beta$ denotes a blunt edge.
\label{def:LIG}
\end{defn}

A local independence graph can be inferred directly from 
$\lambda$ and $\Sigma$, see also Definition \ref{def:CLIG} below. It is 
primarily of interest when it can be used to infer non-trivial results about
additional local independencies. To this end, 
\citet{mogensenUAI2018} show that regular \ito processes with a  diagonal
$\sigma$ satisfy a so-called global 
Markov  
property with respect to their local independence graphs -- assuming
certain integrability constraints are satisfied -- and one can 
read off local independencies from the graph using a straightforward 
algorithm. This allows us to
answer a filtration question: for $D \subseteq [n]$  and $\beta \in 
[n]$, which
coordinates in $D$ does $$E\left(\lambda_t^{\beta} \mid 
\mathcal{F}_t^D\right)$$ 
depend upon? We conjecture that a generalization of the global Markov property
holds for nondiagonal $\sigma$ as well, but this cannot be shown using the
same techniques as in \cite{mogensenUAI2018}. We do, however, show in
Theorem \ref{def:locind} that for a particular class of \ito 
diffusions the global Markov property does in fact hold for the 
{\it canonical} local independence graph that will be defined below. 

\subsection{\ito diffusions}

A regular \ito diffusion is a regular \ito process such that the drift is of 
the form
$$\lambda_t = \lambda(X_t)$$
for a continuous function $\lambda : \mathbb{R}^n \to \mathbb{R}^n$. 
In differential form,

$$\md X_t = \lambda(X_t) \md t + \sigma \md W_t.$$

\ito diffusions with a constant diffusion coefficient are particularly 
interesting examples of \ito processes. They are Markov processes, but 
they are not closed under marginalization and we need to consider the larger 
class of \ito processes to obtain a class which is closed under 
marginalization.

\begin{defn}[Canonical local independence graph] Let $X$ be a regular 
\ito diffusion with a continuously differentiable drift $\lambda : \mathbb{R}^n \to \mathbb{R}^n$ and diffusion
matrix $\Sigma$. The canonical local independence graph is the graph, 
$\mathcal{D}$, with nodes $[n]$  such that for 
all $\alpha,\beta\in [n]$ 
$$\partial_{\alpha} \lambda_{\beta} \neq 0 \ \ \Leftrightarrow \ \  \alpha \rightarrow_\D \beta$$
 and such that for $\alpha \neq \beta$
$$\Sigma_{\alpha \beta} \neq 0 \ \ \Leftrightarrow  \ \ \alpha \unEdge_{\D} \beta.$$
	\label{def:CLIG}
\end{defn}

As $\partial_{\alpha} \lambda_{\beta} = 0$ implies that $\lambda_t^\beta = \lambda_{\beta}((X_t^{\delta})_{\delta \in [n] \setminus\{\alpha\}})$
is $\mathcal{F}_t^{[n] \setminus\{\alpha\}}$-measurable, the following result 
is an immediate consequence of Definitions \ref{def:LIG} and \ref{def:CLIG}.

\begin{prop} \label{prop:simpleobs}
	 The canonical local independence graph is a local independence graph.
\end{prop}

Definition \ref{def:CLIG} gives a simple operational procedure for determining 
the canonical local independence graph for a regular \ito diffusion directly from 
$\lambda$ and $\Sigma$. It is, however, possible that $\lambda_{\beta}$ has a functional 
form that appears to depend on the coordinate $\alpha$, while 
actually $\alpha \not\rightarrow  \beta \mid [n]  \setminus \{
\alpha\}$. In such a case, the canonical local independence graph will  
not be minimal. 

\begin{exmp}[Smoluchowski diffusion] In this example we 
link the notion of local independence and the local independence graph to 
classical undirected graphical models (see, e.g., \cite{lauritzen1996}) for a 
special class of diffusions 
that are widely studied in equilibrium statistical physics. A \emph{ 
Smoluchowski 
diffusion} is a regular \ito diffusion with
$$\lambda(x) = - \nabla V(x)$$ 
for a continuously differentiable function $V : \mathbb{R}^n \to \mathbb{R}$
and $\sigma = \sqrt{2 \tau} I$ for a constant $\tau > 0$. Thus 
the diffusion matrix $\Sigma = 2 \tau I$ is diagonal. The function
$V$ is called the potential and $\tau$ is called a 
temperature parameter. Since the drift is a gradient, the dynamics of a Smoluchowski 
diffusion is a gradient flow perturbed by white noise. If $V(x) \to \infty$
for $\|x\| \to \infty$ and 
$$Z = \int e^{- \frac{1}{\tau}V(x)} \mathrm{d} x < \infty,$$
 the diffusion has the Gibbs measure with density 
$$\pi(x) = \frac{1}{Z} e^{- \frac{1}{\tau} V(x)}$$ 
as equilibrium distribution, see Proposition 4.2 in \cite{Pavliotis:2014}.  
When $V$ is twice differentiable, Definition \ref{def:CLIG} gives the 
canonical local independence graph $\mathcal{D}$ with 
arrows $\alpha \rightarrow_\D \beta$ whenever
$\partial_{\alpha} \lambda_{\beta} = \partial_{\alpha} \partial_{\beta} V \neq 0$. 
Since 
$$\partial_{\alpha} \lambda_{\beta} =  \partial_{\alpha} \partial_{\beta} V = 
 \partial_{\beta} \partial_{\alpha} V = \partial_{\beta} \lambda_{\alpha}$$
the graph $\mathcal{D}$ enjoys the symmetry property that 
$\alpha \rightarrow_\D \beta$ if and only if $\beta \rightarrow_\D \alpha$.
We denote by $\mathcal{G}$ the undirected version of $\mathcal{D}$, i.e., 
$\alpha -_\G \beta$ if and only if $\alpha \rightarrow_\D \beta$ if and only if 
$\beta \rightarrow_\D \alpha$. 
For any $\alpha, \beta \in [n]$ with $\alpha \not -_{\mathcal{G}} \beta$ it 
follows from $\partial_{\alpha} \partial_{\beta} V = 
 \partial_{\beta} \partial_{\alpha} V = 0$ that 
 $$V(x) = V_1(x_{\alpha}, x_{- \{\alpha, \beta\}}) + 
 V_2(x_{\beta}, x_{- \{\alpha, \beta\}})$$
where $x_{- \{\alpha, \beta\}}$ denotes the vector $x$ with coordinates 
$x_{\alpha}$ and $x_{\beta}$ removed. From this decomposition of $V$ we 
see that $\pi$ has the pairwise Markov property with respect to $\mathcal{G}$,
and it follows from the Hammersley-Clifford theorem that $\pi$ factorizes 
according to $\mathcal{G}$. That is, the potential has the following additive 
decomposition
$$V(x) = \sum_{c \in \mathcal{C}(\mathcal{G})} V_c(x_c)$$
where $\mathcal{C}(\mathcal{G})$ denotes the cliques of $\mathcal{G}$. This
establishes a correspondence between local independencies for a 
Smoluchowski diffusion and Markov properties of its equilibrium distribution.
\label{exmp:smoluch}
\end{exmp}

We
emphasize that the link in Example \ref{exmp:smoluch} between local
independencies representing structural constraints on the dynamics  on the one
side and Markov properties of an equilibrium distribution on the other side is 
a consequence of the symmetry of  the drift of Smoluchowski
diffusions combined with the  diffusion matrix being a scalar multiple of the
identity matrix.  For diffusions with a non-gradient drift or with a more
complicated diffusion  matrix the equilibrium distribution may have no
conditional independencies even though there are strong structural constraints 
on
the dynamics of the process which can be expressed in terms of a sparse local
independence  graph. A simple process which can illustrate this is the
Ornstein-Uhlenbeck process.

\begin{exmp}[Ornstein-Uhlenbeck processes] A regular \ito diffusion with drift
$$\lambda(x) = M(x - \mu)$$
for an $n \times n$ matrix $M$ and an $n$-dimensional vector $\mu$ is called a
regular \emph{Ornstein-Uhlenbeck process}. With $\mathcal{D}$ its canonical 
local 
independence graph, $\alpha \rightarrow_\D \beta$ whenever $M_{\beta \alpha} 
\neq 0 $, and $\alpha \not\rightarrow \beta 	\mid [n] \setminus \{\alpha\}$ 
if $M_{\beta \alpha} = 0$. If $M$ is a stable matrix, then the
Ornstein-Uhlenbeck process has an invariant Gaussian distribution
$\mathcal{N}(\mu, \Gamma_{\infty})$ where $\Gamma_{\infty}$ solves the Lyapunov
equation, 
$$M \Gamma_{\infty} + \Gamma_{\infty} M^T + \Sigma = 0,$$
see Proposition 3.5 in \cite{Pavliotis:2014} or Theorem 2.12 in 
\cite{Jacobsen:1993}.

If $M$ is also symmetric, then $\lambda$ is a gradient, and if $\Sigma = 2
\tau I$ we see that the solution of the Lyapunov equation is
$\Gamma_{\infty} = - \tau M^{-1}$, and $\lambda$ is the negative  
gradient of the quadratic potential   
$$V(x) = - \frac{1}{2} (x - \mu)^T M (x - \mu) 
= \frac{\tau}{2} (x - \mu)^T \Gamma_{\infty}^{-1} (x - \mu).$$
Thus the equilibrium distribution is in a Gaussian graphical model represented 
by 
an undirected graph
$\mathcal{G}$ in which the edges are determined by the non-zero entries of
$\Gamma_{\infty}^{-1} = - \frac{1}{\tau} M$. For this Smoluchowski diffusion we
see very explicitly that the edge  $\alpha - \beta$ is in $\mathcal{G}$ if and
only if both $\alpha \rightarrow \beta$ and $\beta \rightarrow \alpha$ are in
$\mathcal{D}$. However, it is  not difficult to
find an asymmetric but stable matrix $M$ such that  $\Gamma_{\infty}^{-1}$ is a
dense matrix, even if $\Sigma = I$, and the canonical local independence graph 
cannot in
general be determined from Markov properties of the invariant distribution. 

For a general $M$ and general $\Sigma$, and with $D \subseteq [n]$, 
it follows that 
\begin{align*}
E\left(\lambda_t^\beta \mid \mathcal{F}^D_t\right) 
& = \sum_{\delta \in V} M_{\beta\delta}\left(E\left( X_t^\delta \mid 
\mathcal{F}^D_t\right) - 
\mu_\delta\right) \\
& = \sum_{\delta \in \mathrm{pa}(\beta)} 
M_{\beta\delta}\left(E\left( X_t^\delta \mid \mathcal{F}^D_t\right) - 
\mu_\delta\right),
\end{align*}
where $\mathrm{pa}(\beta) = \{ \delta \mid M_{\beta\delta} \neq 0\}$ 
denotes the set of parents of $\beta$ in $\mathcal{D}$. Thus determining
if $\alpha \not\rightarrow  \beta \mid C$ (Definition 
\ref{def:locindSingletonA}) amounts to determining
if $$E\left( X_t^\delta \mid \mathcal{F}^C_t\right)$$ are versions of 
$$E\left( X_t^\delta \mid \mathcal{F}^{C \cup \{\alpha\}}_t\right)$$ for 
$\delta \in \mathrm{pa}(\beta)$. In words, this means 
that if we can predict the values of all the processes, $X_t^\delta$ for 
$\delta \in \mathrm{pa}(\beta)$,
that enter into the drift of coordinate $\beta$ just as well 
from the $C$-histories as we can from the $C \cup \{\alpha\}$-histories
then $\beta$ is locally independent of $\alpha$ given $C$.
\label{exmp:ou}
\end{exmp}

The following sections of the paper will develop the graph theory needed to
answer questions about local independence via graphical properties of the local
independence graph. This theory can be applied as long as the processes
considered have the global Markov property, and we show that this is the case 
for regular
Ornstein-Uhlenbeck processes with respect to their canonical local independence 
graphs. 

\section{Directed correlation graphs}
\label{sec:graphLI}

A {\it graph} is a pair $(V,E)$ where $V$ is a set of nodes and $E$ is a set of 
edges. Each node represents a coordinate processes and therefore we will 
let $V = \{1,2\ldots,n \} = [n]$ when we model a stochastic 
process $X = (X_t)_{t\in \mathcal{T}}$ such that $X_t = (X_t^\alpha)_{\alpha\in 
[n]} \in \mathbb{R}^n$. Every edge is between a pair of nodes. Edges can be of 
different 
types. 
In this paper, we will consider {\it directed} edges, $\rightarrow$, {\it 
bidirected} edges, $\leftrightarrow$, and {\it blunt} edges, $\unEdge$.
Let 
$\alpha,\beta \in V$. Note that $\alpha\rightarrow\beta$ and 
$\beta\rightarrow\alpha$ are different edges. We do not distinguish between 
$\alpha\leftrightarrow\beta$ and $\beta\leftrightarrow\alpha$, nor between 
$\alpha\unEdge\beta$ and $\beta\unEdge\alpha$. We allow directed, and 
bidirected loops 
(self-edges), $\alpha\rightarrow\alpha$ and 
$\alpha\leftrightarrow\alpha$, but 
not blunt loops, $\alpha\unEdge\alpha$. If the edge $\alpha \rightarrow \beta$ 
is in a 
graph, then we 
say that $\alpha$ is a {\it parent} of $\beta$ and write $\alpha\in 
\pa(\beta)$. If $\alpha$ and $\beta$ are joined by a 
blunt edge, $\alpha\unEdge\beta$, then we say that they are {\it spouses}. We 
use $\alpha\sim\beta$ to 
symbolize a 
generic edge between $\alpha\in V$ and $\beta\in V$ of any of these three 
types. We say that $\alpha$ and $\beta$ are {\it adjacent} in the graph $\D$ if 
$\alpha \sim \beta$ in $\D$. We use the notation $\alpha\sim_\D\beta$ to 
highlight that the edge is in $\D$ and we use subscript $\sim_j$ to identify 
edges, $j\in\mathbb{N}$. We use $\alpha\starrightarrow\beta$ to 
symbolize that 
either 
$\alpha\rightarrow\beta$ or $\alpha\leftrightarrow\beta$.

\begin{defn}
	Let $\mathcal{D} = (V,E)$ be a graph. We say that $\D$ is a \emph{directed 
	graph} (DG) if every edge is directed. We say that $\D$ 
	is a \emph{directed 
	correlation graph} (cDG) if every edge is directed or blunt. We say 
	that $\D$ is a \emph{directed 
	mixed graph} (DMG) if every $e\in E$ is either directed or bidirected.
\end{defn}

The class of DMGs is studied by \citet{mogensenUAI2018, Mogensen2020a}. 
\citet{eichlerGranger2007, eichlerChapter2012} studies classes of graphs 
similar 
to 
cDGs as well as a class of graphs which contains both the DMGs and the cDGs as 
subclasses. This paper is mostly concerned with the class of cDGs, however, we 
mention the DMGs for two reasons: 1) to compare with the cDGs and demonstrate 
their differences, and 2) to show that the concept of $\mu$-separation can be 
applied to both classes of graphs, and therefore also to a superclass of graphs 
containing both the DMGs and the cDGs. In a cDG, a directed edge corresponds 
to a direct dependence in the drift of the process while a blunt edge 
represents a correlation in the driving Brownian motions (Definition 
\ref{def:LIG}). In a DMG, a directed edge has the same interpretation as in a 
cDG, however, 
a bidirected edge corresponds to a dependence arising from partial observation, 
i.e., marginalization. Correlated driving Brownian motions and 
marginalization create different local independence structures, hence the 
distinction between DMGs and cDGs.

A {\it walk}, $\omega$, is an ordered, alternating sequence of nodes 
($\gamma_i$) and 
edges ($\sim_j$) such that 
each edge, $\sim_i$, is between $\gamma_i$ and $\gamma_{i+1}$,

$$
\gamma_1 \sim_1 \gamma_2 \sim_2 \ldots \sim_{k} \gamma_{k+1}.
$$

\noindent For each directed edge, its orientation is also known as otherwise 
$\alpha\rightarrow\alpha$ and $\alpha\leftarrow\alpha$ would be 
indistinguishable. We say that $\gamma_1$ and $\gamma_{k+1}$ are {\it endpoint 
nodes}, and we say that the walk is {\it from} $\gamma_1$ {\it to} 
$\gamma_{k+1}$. For later purposes this orientation of the walk is essential. 
We let $\omega^{-1}$ denote the walk obtained by traversing the nodes and edges 
of $\omega$ in reverse order. At times, we will also say that a walk, $\omega$, 
is {\it between} $\gamma_1$ and $\gamma_{k+1}$, but only when its orientation 
does not matter in which case we essentially identify $\omega$ with 
$\omega^{-1}$.
We 
say that a walk is {\it trivial} if it has no edges and therefore only a single 
node, and otherwise we say that it is {\it nontrivial}. Consider a walk as 
above. We 
say that an nonendpoint node, $\gamma_i$, ($i\neq 1,k+1$) is a {\it collider} 
if the subwalk

$$
\gamma_{i-1} \sim_{i-1} \gamma_i \sim_i \gamma_{i+1}
$$

\noindent is of one the following types

\begin{align*}
	\gamma_{i-1} \starrightarrow &\gamma_i \starleftarrow \gamma_{i+1}, \\
	\gamma_{i-1} \starrightarrow &\gamma_i \unEdge \gamma_{i+1}, \\
	\gamma_{i-1}  \unEdge &\gamma_i \starleftarrow \gamma_{i+1}, \\
	\gamma_{i-1}  \unEdge &\gamma_i \unEdge \gamma_{i+1},
\end{align*}

\noindent and otherwise we say that it is a {\it noncollider}. This means that 
the property of being a collider or a noncollider is relative to a walk and, 
seeing that nodes may be repeated on a walk, it is actually a property of an 
instance of a node on a specific walk. Note that 
endpoint nodes are neither colliders nor noncolliders. We say that $\alpha$ and 
$\beta$ are {\it collider connected} if 
there exists a walk from $\alpha$ to $\beta$ such that every nonendpoint node 
is a collider.

We say that 
$\alpha \starrightarrow \beta$ has a {\it head} at $\beta$, and that $\alpha 
\rightarrow\beta$ has a {\it tail} at $\alpha$. We say that $\alpha \unEdge 
\beta$ 
has a {\it stump} at $\alpha$. We say that edges $\alpha \unEdge \beta$ and 
$\alpha \starrightarrow \beta$ have a {\it neck} at $\beta$. It follows that 
$\gamma_i$ above is a collider if and only if both adjacent edges have a neck 
at $\gamma_i$. A {\it path} is a walk such that every node occurs at most once. 
We say that a path from $\alpha$ to $\beta$ is {\it directed} if every edge on 
the path is directed and pointing towards $\beta$. 
If there is a directed path from $\alpha$ to $\beta$, then we say that $\alpha$ 
is an {\it ancestor} of $\beta$ and that $\beta$ is a {\it descendant} of 
$\alpha$. We let $\an(\beta)$ denote the set of 
ancestors of $\beta$, and for $C\subseteq V$, we define $\an(C) = 
\cup_{\gamma\in C} \an(\gamma)$. Note that $C\subseteq \an(C)$. A {\it cycle} 
is a path $\alpha \sim \ldots \sim \beta$ composed with an edge 
$\beta\sim\alpha$. If the path from $\alpha$ to $\beta$ is directed and the 
edge is $\beta\rightarrow\alpha$, then we say that the cycle is {\it directed}. 
A DG without any directed cycles is said to be a {\it directed acyclic graph} 
(DAG).

When $\D=(V,E)$ is a graph and $\bar{V}\subseteq V$,  
we let 
$\D_{\bar{V}}$ denote the {\it induced graph} on nodes $\bar{V}$, i.e., 
$\D_{\bar{V}} 
= 
(\bar{V}, \bar{E})$,

$$
\bar{E} = \{e \in E: \text{ $e$ is between $\alpha, \beta \in \bar{V}$}  \}.
$$

We will use 
$\mu$-{\it connecting walks} and 
$\mu$-{\it separation} to encode 
independence structures using cDGs.

\begin{defn}[$\mu$-connecting walk \cite{Mogensen2020a}]
	Consider a nontrivial walk, $\omega$,
	
	$$
	\alpha \sim_1 \gamma_2 \sim_2 \ldots \sim_{k-1} \gamma_k \sim_{k} 
	\beta
	$$
	
	\noindent and a set $C\subseteq V$. We say that $\omega$ is  
	\emph{$\mu$-connecting} from $\alpha$ to $\beta$ given $C$ if $\alpha\notin 
	C$, 
	every 
	collider on $\omega$ is in $\an(C)$, no noncollider is in $C$, and 
	$\sim_{k}$ has a head at $\beta$.
	
\end{defn}

It is essential that the above definition uses walks, 
and not only paths. As an example consider $\alpha \unEdge \beta \leftarrow 
\gamma$. In this graph, 
there is no $\mu$-connecting path from $\alpha$ to $\beta$ given $\beta$, but 
there is a $\mu$-connecting walk.

\begin{defn}[$\mu$-separation \cite{Mogensen2020a}]
	Let $\G = (V,E)$ be a cDG or a DMG and let $A,B,C\subseteq V$. We say that 
	$B$ is 
	\emph{$\mu$-separated} from $A$ given $C$ in $\G$ if there is no 
	$\mu$-connecting walk from any $\alpha\in A$ to any $\beta\in B$ given $C$ 
	and we denote this by $\musepG{A}{B}{C}{\G}$, or just $\musep{A}{B}{C}$.
	\label{def:musep}
\end{defn}

When sets $A$, $B$, or $C$ above are singletons, e.g., $A = \{\alpha\}$, we 
write $\alpha$ instead of $\{\alpha\}$ in the context of $\mu$-separation. 
\citet{Mogensen2020a} introduced $\mu$-separation as a generalization of 
$\delta$-separation \citep{didelez2000, didelez2008}, however, only in DMGs, 
and not in cDGs. As in other classes of graphs, one can decide $\mu$-separation 
in cDGs by using an auxiliary 
undirected graph. 
This is described in 
Appendix \ref{app:aug}.

When 
 $\mathcal{D} = (V,E)$ is a cDG or DMG, we define its {\it independence model} 
 (or {\it separation model}), 
 $\I(\mathcal{D})$, as the collection of $\mu$-separations that hold, i.e.,
 
 $$
 \I(\mathcal{D}) = \{ (A,B,C): A,B,C \subseteq V,\ 
 \musepG{A}{B}{C}{\mathcal{D}} 
 \}.
 $$
 
 \begin{defn}[Markov equivalence]
 	Let $\D_1 = (V,E_1)$ be a cDG or a DMG and let $\D_2 = (V,E_2)$ be a cDG or 
 	a 
 	DMG. We say that $\D_1$ 
 	and 
 	$\D_2$ are \emph{Markov equivalent} if $\I(\D_1) = \I(\D_2)$. 
 \end{defn}
 
 For any finite set $V$, Markov equivalence
 is an 
 equivalence relation on a set of graphs with node set $V$. When $\D$ 
 is a cDG or a DMG, we let 
 $[\D]$ 
 denote 
 the Markov equivalence class of $\D$ {\it restricted to its own class 
 of graphs}. That is, if $\D$ is a cDG, then $[\D]$ denotes the set of Markov 
 equivalent cDGs. If $\D$ is a DMG, then $[\D]$ denotes the set of Markov 
 equivalent DMGs. For a cDG or DMG, $\D = 
 (V,E)$, and a directed, blunt, or bidirected edge $e$ between $\alpha \in V$ 
 and $\beta\in 
 V$, 
 we 
 use $\D + e$ to denote the graph $(V, E\cup \{e \})$.
 
 \begin{defn}[Maximality]
 	Let $\D=(V,E)$ be a cDG (DMG). We say that $\D$ is \emph{maximal} if no 
 	directed or blunt (directed or bidirected) 
 	edge can 
 	be 
 	added Markov equivalently, i.e., if for every directed or blunt (directed 
 	or bidirected) edge, $e$, such that
 	$e\notin E$, it holds that
 	$\D$ and $\D + e$ are not Markov equivalent.
 	\label{def:max}
 \end{defn}

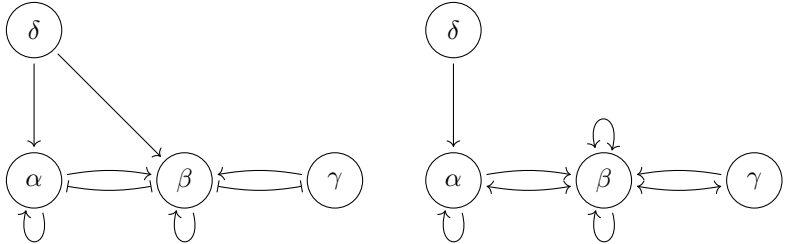
\begin{figure}
	\begin{minipage}{0.45\linewidth}
		\begin{tikzpicture}
		\node[shape=circle,draw=black] (a) at (0,0) {$\alpha$};
		\node[shape=circle,draw=black] (b) at (2,0) {$\beta$};
		\node[shape=circle,draw=black] (c) at (4,0) {$\gamma$};
		\node[shape=circle,draw=black] (d) at (0,2) {$\delta$};
		
		\path [->] (a) edge [bend left = 10] node {} (b);
		\path [->] (d) edge [bend left = 0] node {} (a);
		\path [->] (d) edge [bend left = 0] node {} (b);
		\path [|-|] (a) edge [bend right = 10] node {} (b);
		\path [|-|](b) edge [bend right = 10] node {} (c);
		\path [<-](b) edge [bend left = 10] node {} (c);
		\draw[every loop/.append style={->}] (b) edge[loop below] 
		node {}  (b);
		\draw[every loop/.append style={->}] (a) edge[loop below] 
		node {}  (a);	
		\end{tikzpicture}
	\end{minipage}
	\begin{minipage}{0.45\linewidth}
		\begin{tikzpicture}
		\node[shape=circle,draw=black] (a) at (0,0) {$\alpha$};
		\node[shape=circle,draw=black] (b) at (2,0) {$\beta$};
		\node[shape=circle,draw=black] (c) at (4,0) {$\gamma$};
		\node[shape=circle,draw=black] (d) at (0,2) {$\delta$};
		
		\path [->] (a) edge [bend left = 10] node {} (b);
		\path [->] (d) edge [bend left = 0] node {} (a);
		\path [<->] (a) edge [bend right = 10] node {} (b);
		\path [<->](b) edge [bend right = 10] node {} (c);
		\path [<-](b) edge [bend left = 10] node {} (c);
		\draw[every loop/.append style={->}] (b) edge[loop below] 
		node {}  (b);
		\draw[every loop/.append style={->}] (a) edge[loop below] 
		node {}  (a);	
		\draw[every loop/.append style={<->}] (b) edge[loop above] 
		node {}  (b);
		\end{tikzpicture}
		
	\end{minipage}
	\caption{Example cDG (left) and example DMG 
	(right). The blunt edges 
	in a cDG correspond to correlated driving processes which is different from 
	the bidirected edges of a DMG as those correspond to marginalization, i.e., 
	unobserved processes. The notion of $\mu$-separation can be applied to both 
	classes of graphs. Left: cDG 
	on 
	nodes $V=\{\alpha,\beta,\gamma,\delta \}$. $\gamma$ is $\mu$-separated 
	(Definition \ref{def:musep}) from 
	$\delta$ by $\alpha$ as $\beta\notin \an(\alpha)$ is a collider on any walk 
	from $\delta$ to $\gamma$. On the other hand, $\alpha$ is not 
	$\mu$-separated from $\beta$ 
	given $\emptyset$ as, e.g., $\beta 
	\unEdge 
	\alpha \rightarrow\alpha$ is 
	$\mu$-connecting given $\emptyset$. The same walk is not $\mu$-connecting 
	from 
	$\beta$ to $\alpha$ given $\alpha$, however, $\beta \leftarrow 
	\delta\rightarrow\alpha$ is $\mu$-connecting from $\beta$ to $\alpha$ given 
	$\alpha$. We see that $\alpha$ is $\mu$-separated from $\beta$ 
	given $\{\alpha,\delta \}$. Right: bidirected edges have heads at both ends 
	and this means that $\beta \leftrightarrow \alpha$ is $\mu$-connecting from 
	$\beta$ to $\alpha$ given any subset of $V\setminus \{\beta\}$. This is not 
	true in the cDG (left).}
	\label{fig:cDMGsExmps}
\end{figure}

\begin{rem}
	\citet{eichlerGranger2007, eichler2010}, and \citet{eichler2012} describe 
	graphs that 
	represent local independence (or 
	Granger non-causality) in time series in the presence of correlated error 
	processes. In those papers, blunt edges are 
	represented by (dashed or solid) undirected edges, $-$, while we use 
	$\unEdge$. 
	The former 
	representation 
	could suggest that a blunt edge acts like an edge with tails in both 
	ends, however, 
	this is not the case. It also does not act like the bidirected edges in a 
	DMG, and this warrants the usage of an edge with a third kind of mark.
	
	Also note that while this is not paramount in the case of cDGs, notational 
	clarity and simplicity
	become more important when considering graphical marginalizations of these 
	graphs. In this case, one needs to consider also edges that, when composed 
	with other edges, act like a 
	blunt edge in one end and like a directed edge in the other (see 
	also \cite{eichlerChapter2012}). Using our 
	representation, this can be visualized by the edge $\mapsto$. We 
	will 
	not 
	consider this larger class of graphs in this paper, however, we choose this 
	representation as it extends naturally.
\end{rem}

\subsection{A global Markov property}

In this section, we state a result showing that an Ornstein-Uhlenbeck process, 
$X$, 
satisfies a 
global Markov property with respect to a cDG, $\D$, when $\D$ is the canonical 
local 
independence graph of $X$ (Definition \ref{def:CLIG}). As we identify the 
coordinate processes of $X$ with 
nodes in $\D$, we use 
$V$ to denote both the node set of $\D$ and as the index set of the 
coordinate 
processes of $X$. In the 
case of a diagonal $\Sigma$, the global Markov property was shown in 
\cite{mogensenUAI2018}, 
and we extend this to the case of nondiagonal $\Sigma$, i.e., allowing 
for correlated driving Brownian motions. Before stating the result, 
we give a more 
general definition of local independence in \ito processes to allow 
non-singleton sets $A$ and $B$.

\begin{defn} Let $X$ be a regular \ito process with drift $\lambda$,
	and let $A,B,C \subseteq V$. 
	We say that $B$ is locally independent of $A$ given $C$, 
	and write $A \not\rightarrow  B \mid C$, if for all $\beta\in B$ the 
	process 
	$$ t \mapsto E\left(\lambda^{\beta}_t \mid \mathcal{F}_t^{C}\right)$$
	is a version of $$t \mapsto E\left(\lambda^{\beta}_t \mid \mathcal{F}_t^{C 
	\cup 
	A}\right).$$
	\label{def:locind}
\end{defn}

\begin{thm}
	Let $X = (X_t)_{t \geq 0}$ be a regular Ornstein-Uhlenbeck process, let 
	$\D$ be its canonical  
	local independence graph (Definition 
	\ref{def:CLIG}), and let
	$A,B,C \subseteq V$. Assume that  
	$X_0$ is a (non-degenerate) multivariate Gaussian vector with independent 
	entries and that $X_0$ is independent of the Brownian motion driving the 
	Ornstein-Uhlenbeck process. If 
	$B$ is $\mu$-separated from $A$ given $C$ in 
	$\mathcal{D}$, 
	then $B$ is locally independent of $A$ given $C$.
	\label{thm:globalMarkov}
\end{thm}

 The result allows us to infer sparsity in the dependence structure in the 
 evolution of the process from structural sparsity encoded by a cDG 
 and 
 $\mu$-separation. The proof of Theorem \ref{thm:globalMarkov} is found in 
 Appendix 
 \ref{app:proofMarkov} and it uses a set of 
 equations describing the
 conditional mean processes, $t \mapsto \E{X_t^U}{\mathcal{F}_t^W}$, $V = 
 U\disjU W$, see \cite{liptser1977}. From this representation, one can reason 
 about 
 the measurability of the conditional mean processes. 

The global Markov property can be seen to be somewhat similar to that of chain 
graphs under the MVR interpretation \cite{cox1993, sonntag2015} (see also 
\cite{koster1999}).

\section{Markov equivalence of directed correlation graphs}
\label{sec:ME}

Different cDGs can encode the same separation model and in this section we 
will describe the Markov equivalence classes of cDGs. This is essential as it 
allows us to understand which graphical structures 
represent the same local independencies. This understanding is needed if we 
want 
to learn graphical representations from tests of local independence in observed 
data. We begin this section by noting a strong link between the independence 
model 
of a cDG and 
its directed edges.

\begin{prop}
	Let $\mathcal{D} = (V,E)$ be a cDG. Then 
	$\alpha\rightarrow_{\mathcal{D}}\beta$ if and only if 
	$\musep{\alpha}{\beta}{V\setminus \{\alpha\}}$ does not hold.
	\label{prop:dirEdgeSep}
\end{prop}

The proposition can be found in \cite{Mogensen2020a} in the case of 
DGs. Proposition \ref{prop:dirEdgeSep} implies that if $\D_1$ and $\D_2$ are 
Markov equivalent cDGs, then they have the same directed edges, and therefore 
$\an_{\D_1}(C) = \an_{\D_2}(C)$ for all node sets $C$. We will often omit the 
subscript when it is clear from the context from which graph(s) the ancestry 
should read off.

\begin{proof}
	If the edge is in the graph, it is $\mu$-connecting given any subset of $V$ 
	that does not contain $\alpha$, in particular given $V\setminus 
	\{\alpha\}$. On the other hand, assume $\alpha \rightarrow \beta$ is not in 
	the graph. Any $\mu$-connecting walk from $\alpha$ to $\beta$ must have a 
	head at $\beta$,
	
	$$
	\alpha \sim \ldots \sim \gamma \rightarrow \beta.
	$$
	
	\noindent We must have that $\gamma \neq \alpha$, and it follows that 
	$\gamma$ is in the conditioning set, i.e., the walk is closed.
\end{proof}

In graphs that represent conditional independence in multivariate 
distributions, such as ancestral graphs and acyclic directed mixed graphs, one 
can use {\it inducing paths} to 
characterize which nodes cannot be separated by any conditioning set 
\citep{vermaEquiAndSynthesis, richardson2002}. In DMGs, inducing 
paths can be defined similarly \citep{Mogensen2020a}. In cDGs, we define 
both inducing paths and {\it weak inducing paths}. We say that a path is a {\it 
collider path} if every nonendpoint node on the path is a collider. If 
$\alpha\neq \beta$, then $\alpha\rightarrow\beta$ and $\alpha\unEdge\beta$ are 
both collider paths.

\begin{defn}[Inducing path (strong)]
	A (nontrivial) collider path from $\alpha$ to $\beta$ is a \emph{(strong) 
	inducing path} 
	if the 
	final edge has a head at $\beta$ and every nonendpoint node is an ancestor 
	of $\alpha$ 
	or of $\beta$.
	\label{def:IP}
\end{defn}

\citet{Mogensen2020a} also allow cycles in the definition of 
inducing paths. In the following, we assume that $\alpha\rightarrow\alpha$ for 
all $\alpha\in V$ and therefore this would be an unnecessary complication. We 
see immediately that in a cDG, the only inducing path is a directed 
edge. However, we include this 
definition to conform with the terminology in DMGs where more elaborate 
inducing paths exist. If we drop one of the conditions from Definition 
\ref{def:IP}, then we obtain a graphical structure which is more interesting in 
cDGs, a 
{\it weak inducing path}.

\begin{defn}[Weak inducing path]
	A (nontrivial) collider path between $\alpha$ and $\beta$ is a \emph{weak 
	inducing path} 
	if every nonendpoint node is an ancestor of $\alpha$ or $\beta$.
	\label{def:wIP}
\end{defn}

\noindent We note that a strong inducing path is also a weak inducing path. 
Furthermore, if there is a weak inducing path from $\alpha$ to $\beta$, there 
is also one from $\beta$ to $\alpha$, and this justifies saying that a weak 
inducing path is {\it between} $\alpha$ and $\beta$ in Definition 
\ref{def:wIP}. Also note that a {\it weak} inducing path 
is most often called an inducing path in the literature on acyclic graphs. When 
we just say {\it inducing path}, we mean a strong inducing path.

If $\D$ is a cDG such that 
$\alpha\rightarrow_{\D}\alpha$ for 
all $\alpha \in V$, then we say that $\D$ {\it contains every loop}. From this 
point on, we will 
assume that the cDGs we consider all contain 
every loop.

\begin{prop}
	Let $\D = (V,E)$ be a cDG such that $\alpha\rightarrow\alpha$ for all 
	$\alpha\in V$. 
	There is a weak inducing path between $\alpha$ and $\beta$ if and only if 
	there 
	is no $C\subseteq V\setminus \{\alpha,\beta\}$ such that 
	$\musep{\alpha}{\beta}{C}$.
	\label{prop:inSepWIP}
\end{prop}
 
\citet{Mogensen2020a} show a similar result in the case of strong inducing 
paths 
in DMGs.

\begin{proof}
	Assume first that there is no weak inducing path between $\alpha$ and 
	$\beta$ 
	in $\D$, and define
	
	$$
	D(\alpha,\beta) = \{ \gamma\in \an(\alpha,\beta) \mid \gamma\text{ and } 
	\beta \text{ are collider connected }   \} 	\setminus 
	\{\alpha,\beta\}.
	$$
	
	\noindent We will show that $\beta$ is $\mu$-separated from $\alpha$ by 
	$D(\alpha,\beta)$. We can assume that $\alpha\neq\beta$ as we 
	have assumed 
	that all nodes have loops. If there is a 
	$\mu$-connecting walk from $\alpha$ to $\beta$ given $C\subseteq V\setminus 
	\{\alpha,\beta\}$, then there is also a $\mu$-connecting walk which is a 
	path composed with a directed edge, $\gamma\rightarrow\beta$. We must have 
	that $\gamma\neq \alpha$, and if $\gamma \neq \beta$ then the walk is 
	closed by $D(\alpha,\beta)$. Assume instead that $\gamma = \beta$. 
	Let $\pi$ 
	denote some path between $\alpha$ and $\beta$. Blunt and 
	directed edges are weak inducing paths (in either direction) so $\pi$ 
	must be of length 2 or more,
	
	$$
	\alpha = \gamma_0 \overset{e_0}{\sim} \gamma_1 \overset{e_1}{\sim} \ldots  
	\overset{e_{j-1}}{\sim} \gamma_j \overset{e_{j}}{\sim} \beta.
	$$
	
	\noindent There must exist $i \in \{0,1,\ldots,j \}$, $j \geq 1$, such that 
	either 
	$\gamma_i$ is not 
	collider connected to $\beta$ along $\pi$ or $\gamma_i \notin 
	\an(\alpha,\beta)$. Let $i_+$ denote the largest such number in 
	$\{0,1,\ldots,j \}$. Assume first that $\gamma_{i_+}$ is not 
	collider connected to 
	$\beta$ along $\pi$. In this case, $i_+ \neq j$. Then $\gamma_{i_+ + 1}$ is 
	a 
	noncollider on $\pi$ and it is in $D(\alpha,\beta)$, and it follows that 
	$\pi$ is not $\mu$-connecting. Note that necessarily $\gamma_{i_+ +1} \neq 
	\alpha,\beta$. On the other hand, assume $\gamma_{i_+}\notin 
	\an(\alpha,\beta)$. Then $i_+ \neq 0$, and there is some collider, 
	$\gamma_k$, on $\pi$, 
	$k\in\{1,...,i_+\}$. We have that $\gamma_k\notin \an(\alpha,\beta)$ and 
	$\pi$ is closed in this collider.
	
	On the other hand, assume that there is a weak inducing path between 
	$\alpha$ and $\beta$ and let $C\subseteq V\setminus \{\alpha,\beta\}$. If 
	$\alpha=\beta$, then $\alpha\rightarrow\beta$ 
	which is connecting given $C$. Assume 
	$\alpha\neq \beta$. If $\alpha$ and $\beta$ are adjacent, then 
	$\alpha\sim\beta\rightarrow\beta$ is $\mu$-connecting given $C\subseteq 
	V\setminus\{\alpha,\beta\}$. Consider the weak inducing path,
	
	$$
	\alpha \sim \gamma_1 \sim \ldots \gamma_j \sim \beta = \gamma_{j+1}.
	$$
	
	\noindent Let $k$ be the maximal number in the set $\{1,\ldots,j\}$ such 
	that there is a walk between $\alpha$ and $\gamma_k$ with all colliders in 
	$\an(C)$, no noncolliders in $C$, and which has a neck at $\gamma_k$. We 
	see that $\gamma_1\neq \beta$ fits this description, i.e., $k$ is 
	well-defined. Let $\omega$ be the walk from $\alpha$ to $\gamma_k$. If 
	$\gamma_k \in \an(C)$, then the composition of $\omega$ with $\gamma_k \sim 
	\gamma_{k+1}$ gives either a new such walk (if the edge is blunt) and 
	by maximality of $k$ we have that $\gamma_{k+1} = \beta$, or if the edge is 
	directed 
	then also $\gamma_{k+1}=\beta$ (the weak inducing path is a 
	collider path), and composing either walk with 
	$\beta\rightarrow 
	\beta$ 
	gives a connecting walk given $C$. Assume instead that $\gamma_k\notin 
	\an(C)$, and 
	consider again $\omega$. There is a directed path from $\gamma_k$ to 
	$\alpha$ or to $\beta$. Let $\bar{\pi}$ denote the subpath from $\gamma_k$ 
	to the 
	first instance of either $\alpha$ or $\beta$. If $\alpha$ occurs first, we 
	compose 
	$\bar{\pi}^{-1}$ with $\gamma_k\sim \gamma_{k+1}$ and argue as 
	in the 
	case of $\gamma_j\in \an(C)$ above. In $\beta$ occurs first, $\omega$ 
	composed 
	with $\bar{\pi}$ is connecting. 
\end{proof}

We say that $\beta$ is {\it inseparable} from $\alpha$ if there is no 
$C\subseteq V\setminus \{\alpha\}$ such that $\beta$ is $\mu$-separated from 
$\alpha$ by $C$.

\begin{exmp}
	\citet{Mogensen2020a} use $\mu$-separation in \emph{directed mixed graphs} 
	(DMGs) to represent local independence models. It is natural to ask if the 
	independence models of cDGs can be represented by DMGs. The answer is no 
	and to show this we consider the cDG in Figure 
	\ref{fig:cDGnoDMG} and ask if there exists a DMG on the same node set which 
	has the same set of $\mu$-separations as the cDG in Figure 
	\ref{fig:cDGnoDMG}. In the cDG, we see that the node 
	$\gamma$ is separable from $\alpha$ and vice versa, i.e., there can be no 
	edge between the two in the DMG. The node $\gamma$ is not separated from 
	$\alpha$ given $\{\beta\}$, and therefore $\beta$ must be a collider on a 
	path 
	between the two. However, then there is a head at $\beta$ on an edge from 
	$\gamma$ 
	and 
	therefore $\beta$ is inseparable from $\gamma$ which is a contradiction. 
	This shows that the independence model of the cDG in Figure 
	\ref{fig:cDGnoDMG} cannot be represented by 
	any DMG on the same node set. It follows that the set of separation models 
	of cDGs on some 
	node set is not in general a subset of the separation models of DMGs on the 
	same node set.

	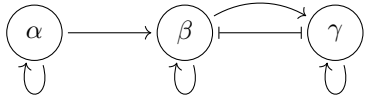
\begin{figure}
		\centering
				\begin{tikzpicture}
				\node[shape=circle,draw=black,style=myCircle] (a) at (0,0) 
				{$\alpha$};
				\node[shape=circle,draw=black,style=myCircle] (b) at (2,0) 
				{$\beta$};
				\node[shape=circle,draw=black,style=myCircle] (c) at (4,0) 
				{$\gamma$};
				
				\path [->] (a) edge [bend left = 0] node {} (b);
				\path [|-|](b) edge [bend right = 0] node {} (c);
				\path [->](b) edge [bend left = 30] node {} (c);
				\draw[every loop/.append style={->}] (b) edge[loop below] 
				node {}  (b);
				\draw[every loop/.append style={->}] (a) edge[loop below] 
				node {}  (a);
				\draw[every loop/.append style={->}] (c) edge[loop below] 
				node {}  (c);
				\end{tikzpicture}
				\caption{A cDG, $\mathcal{D}$, on nodes 
				$V=\{\alpha,\beta,\gamma\}$ such that the separation model 
				$\I(\mathcal{D})$ cannot be represented by a DMG on nodes $V$. 
				See Example 
				\ref{exmp:cDGnoDMG}.}
				\label{fig:cDGnoDMG}
	\end{figure}
	
	\label{exmp:cDGnoDMG}
\end{exmp}

	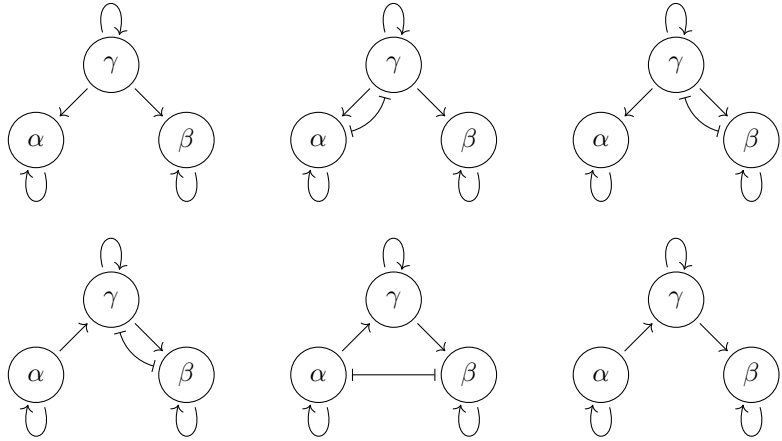
\begin{figure}
		\centering
		\begin{minipage}{0.3\linewidth}
	\begin{tikzpicture}
	\node[shape=circle,draw=black,style=myCircle] (a) at (0,0) {$\alpha$};
	\node[shape=circle,draw=black,style=myCircle] (b) at (2,0) {$\beta$};
	\node[shape=circle,draw=black,style=myCircle] (c) at (1,1) {$\gamma$};
	
	\path [->] (c) edge [bend left = 0] node {} (a);
	\path [<-](b) edge [bend left = 0] node {} (c);
	\draw[every loop/.append style={->}] (b) edge[loop below] 
	node {}  (b);
	\draw[every loop/.append style={->}] (a) edge[loop below] 
	node {}  (a);	
	\draw[every loop/.append style={->}] (c) edge[loop above] 
	node {}  (c);
	\end{tikzpicture}
		\end{minipage}
\begin{minipage}{0.3\linewidth}
	\begin{tikzpicture}
	\node[shape=circle,draw=black] (a) at (0,0) {$\alpha$};
	\node[shape=circle,draw=black] (b) at (2,0) {$\beta$};
	\node[shape=circle,draw=black] (c) at (1,1) {$\gamma$};
	
	\path [|-|] (c) edge [bend left = 30] node {} (a);	
	\path [->] (c) edge [bend left = 0] node {} (a);
	\path [<-](b) edge [bend left = 0] node {} (c);
	\draw[every loop/.append style={->}] (b) edge[loop below] 
	node {}  (b);
	\draw[every loop/.append style={->}] (a) edge[loop below] 
	node {}  (a);	
	\draw[every loop/.append style={->}] (c) edge[loop above] 
	node {}  (c);	
	\end{tikzpicture}
\end{minipage}
\begin{minipage}{0.3\linewidth}
	\begin{tikzpicture}
	\node[shape=circle,draw=black] (a) at (0,0) {$\alpha$};
	\node[shape=circle,draw=black] (b) at (2,0) {$\beta$};
	\node[shape=circle,draw=black] (c) at (1,1) {$\gamma$};
	
	\path [|-|] (c) edge [bend right = 30] node {} (b);
	\path [->] (c) edge [bend left = 0] node {} (a);
	\path [<-](b) edge [bend left = 0] node {} (c);
	\draw[every loop/.append style={->}] (b) edge[loop below] 
	node {}  (b);
	\draw[every loop/.append style={->}] (a) edge[loop below] 
	node {}  (a);	
	\draw[every loop/.append style={->}] (c) edge[loop above] 
	node {}  (c);	
	\end{tikzpicture}
\end{minipage}
		\begin{minipage}{0.3\linewidth}
			\begin{tikzpicture}
			\node[shape=circle,draw=black] (a) at (0,0) {$\alpha$};
			\node[shape=circle,draw=black] (b) at (2,0) {$\beta$};
			\node[shape=circle,draw=black] (c) at (1,1) {$\gamma$};
			
			\path [<-] (c) edge [bend left = 0] node {} (a);
			\path [<-](b) edge [bend left = 0] node {} (c);
			\path [|-|](b) edge [bend left = 30] node {} (c);
			\draw[every loop/.append style={->}] (b) edge[loop below] 
			node {}  (b);
			\draw[every loop/.append style={->}] (a) edge[loop below] 
			node {}  (a);	
			\draw[every loop/.append style={->}] (c) edge[loop above] 
			node {}  (c);
			\end{tikzpicture}
		\end{minipage}
		\begin{minipage}{0.3\linewidth}
			\begin{tikzpicture}
			\node[shape=circle,draw=black] (a) at (0,0) {$\alpha$};
			\node[shape=circle,draw=black] (b) at (2,0) {$\beta$};
			\node[shape=circle,draw=black] (c) at (1,1) {$\gamma$};
			
			\path [|-|] (a) edge [bend left = 0] node {} (b);	
			\path [<-] (c) edge [bend left = 0] node {} (a);
			\path [<-](b) edge [bend left = 0] node {} (c);
			\draw[every loop/.append style={->}] (b) edge[loop below] 
			node {}  (b);
			\draw[every loop/.append style={->}] (a) edge[loop below] 
			node {}  (a);	
			\draw[every loop/.append style={->}] (c) edge[loop above] 
			node {}  (c);	
			\end{tikzpicture}
		\end{minipage}
		\begin{minipage}{0.3\linewidth}
			\begin{tikzpicture}
			\node[shape=circle,draw=black] (a) at (0,0) {$\alpha$};
			\node[shape=circle,draw=black] (b) at (2,0) {$\beta$};
			\node[shape=circle,draw=black] (c) at (1,1) {$\gamma$};
			
			\path [<-] (c) edge [bend left = 0] node {} (a);
			\path [<-](b) edge [bend left = 0] node {} (c);
			\draw[every loop/.append style={->}] (b) edge[loop below] 
			node {}  (b);
			\draw[every loop/.append style={->}] (a) edge[loop below] 
			node {}  (a);	
			\draw[every loop/.append style={->}] (c) edge[loop above] 
			node {}  (c);	
			\end{tikzpicture}
		\end{minipage}

		\caption{First row: an equivalence class illustrating that a greatest 
		element need not exist. Second row: the left and center graphs 
		are Markov equivalent. The graph on the right is the largest graph 
		which is a 
		subgraph of both of them, and this graph is not Markov equivalent, 
		i.e., the Markov equivalence class of the left (and center) graph does 
		not have a least element. Theorem \ref{thm:meChar} gives a 
		characterization of Markov equivalence of cDGs.}
		\label{fig:noMaxNoMinExmp}
	\end{figure}

DGs constitute a subclass of cDGs and within the class 
of DGs every Markov equivalence class is a singleton, i.e., two DGs are Markov 
equivalent if and only if they are equal.

\begin{prop}[Markov equivalence of DGs \cite{Mogensen2020a}]
	Let $\mathcal{D}_1 = (V,E_1)$ and $\mathcal{D}_2 = (V,E_2)$ be DGs. Then 
	$\mathcal{D}_1 \in [\mathcal{D}_2]$ if and only if $\mathcal{D}_1 = 
	\mathcal{D}_2$.  
	\label{prop:indIDsGraph}
\end{prop}

Proposition \ref{prop:indIDsGraph} does not 
hold in general when $\mathcal{D}_1$ and $\mathcal{D}_2$ are cDGs. As an 
example, consider a graph on nodes $\{\alpha,\beta\}$ such that 
$\alpha\rightarrow\beta$ and 
	$\beta\rightarrow\alpha$. This graph is Markov equivalent with the graph 
	where $\alpha\unEdge\beta$ is added. The next result is an immediate 
	consequence of Proposition 
\ref{prop:dirEdgeSep} and shows that Markov equivalent cDGs always have the 
same directed edges.

\begin{cor}
	Let $\mathcal{D}_1 = (V,E_1)$ and $\mathcal{D}_2 = (V,E_2)$ be cDGs. If 
	they are Markov equivalent, then for all $\alpha,\beta\in V$ it holds that 
	$\alpha\rightarrow_{\mathcal{D}_1}\beta$ if and only if 
	$\alpha\rightarrow_{\mathcal{D}_2}\beta$.
	\label{cor:IDunCorr}
\end{cor}

For graphs $\D_1 = (V,E_1)$ and $\D_2 = (V,E_2)$, we write 
$\D_1\subseteq \D_2$ 
if $E_1\subseteq E_2$. We say that a graph, $\D$, is a {\it greatest} element 
of 
its equivalence class, 
$[\D]$, if it is a supergraph of all members of the class, i.e., $\tilde{\D} 
\subseteq \D$ for all $\tilde{\D} \in [\D]$. We say that $\D$ is a {\it 
least} 
element if $\D \subseteq \tilde{\D}$ for all $\tilde{\D}\in [\D]$. 
\citet{Mogensen2020a} show the below result on Markov equivalence. 

\begin{thm}[Greatest Markov equivalent DMG \cite{Mogensen2020a}]
	Let $\G$ be a directed mixed graph. Then $[\G]$  
	has a greatest 
	element (within the class of DMGs), i.e., there exists $\bar{\G}\in [\G]$ 
	such that $\bar{\G}$ is a 
	supergraph of all Markov equivalent DMGs.
\end{thm}

The theorem provides a concise and intuitive way to understand sets of Markov 
equivalent DMGs. If $\G$ is a DMG, then we can visualize $[\G]$ by drawing its 
greatest element and simply showing which edges are in every DMG in $[\G]$ 
and which are only in some DMGs in $[\G]$. cDGs 
represent local independencies allowing for correlation in the driving error 
processes 
and one can ask if the same result on Markov equivalence holds in 
this class of graphs. The answer is in the negative as illustrated by the 
following example.

\begin{exmp}
	Consider the graph to the left on the first row of Figure 
	\ref{fig:noMaxNoMinExmp}. The 
	edge $\alpha \unEdge \gamma$ can be added Markov 
	equivalently and the edge $\beta\unEdge\gamma$ can be added Markov 
	equivalently 
	(center and right graphs), 
	but they cannot both be added Markov equivalently at the same time. This 
	shows that the equivalence class of this graph does not contain a greatest 
	element. Figure \ref{fig:noMaxNoMinExmp} also gives an example showing that 
	an equivalence class of cDGs does not 
	necessarily contain a least element.
	\label{exmp:noMax}
\end{exmp}

\subsection{A characterization of Markov equivalence of cDGs}

When we have a global Markov property, such as the one in Theorem 
\ref{thm:globalMarkov}, the $\mu$-separations of a cDG imply local 
independencies in the distribution of the stochastic process. We saw in Figure 
\ref{fig:noMaxNoMinExmp} that different cDGs may represent the same 
$\mu$-separations 
and it is therefore important to understand which cDGs are equivalent in terms 
of the $\mu$-separations that they entail, that is, are Markov equivalent. The 
central result of this section is a characterization 
of Markov equivalence of cDGs. We define {\it 
collider equivalence} of graphs as a first step in stating this result.

\begin{defn}
	Let $\D_1 = (V,E_1)$, $\D_2 = (V,E_2)$ be cDGs with the same directed 
	edges, and let $\omega$ be a 
	(nontrivial) collider path in $\D_1$,
	
	$$
	\alpha \sim \gamma_1 \sim \ldots \sim \gamma_{k_1} \sim \beta.
	$$
	
	\noindent We say that $\omega$ is \emph{covered} in $\D_2$ if there exists  
	a (nontrivial) collider path in $\D_2$
	
	$$
	\alpha \sim \bar{\gamma}_1 \sim \ldots \sim \bar{\gamma}_{k_2} \sim \beta
	$$
	
	\noindent such that for each $\bar{\gamma}_j$ we have
	$\bar{\gamma}_j \in 
	\an(\alpha,\beta)$ or $\bar{\gamma}_j \in \cup_i \an(\gamma_i)$.
	\label{def:collCover}
\end{defn}

In the above definition $\{\gamma_j \}$ and $\{\bar{\gamma}_j \}$ 
may be the empty set, corresponding to $\alpha$ and $\beta$ being adjacent, 
$\alpha\sim \beta$.  One should also note that a single edge, 
$\alpha \sim \beta$, constitutes a collider path between $\alpha$ and $\beta$ 
(when $\alpha\neq \beta$) 
and that a single edge covers any collider path between $\alpha$ and $\beta$ as 
it has no nonendpoint 
nodes. When $\D_1$ and $\D_2$ have the same directed edges it 
holds that $\an_{\D_1}(C) = \an_{\D_2}(C)$ for all $C \subseteq V$ and 
therefore one can read off the ancestry of $\alpha$, $\beta$, and 
$\{\gamma_i\}$ in either of the graphs in the above definition.

\begin{defn}[Collider equivalence]
	Let $\D_1$ and $\D_2$ be cDGs on the same node set and with the same 
	directed edges. We say that $\D_1$ and 
	$\D_2$ are \emph{collider equivalent} if every collider path in $\D_1$ is 
	covered in $\D_2$ and every collider path in $\D_2$ is covered in $\D_1$.
	\label{def:collEq}
\end{defn}

In the context of collider equivalence, it is 
important to use the convention that every 
node is an ancestor of itself, i.e., $\gamma \in \an(\gamma)$ for all $\gamma 
\in V$. 
Otherwise, a graph would not necessarily be collider equivalent 
with itself. Using this convention, it follows immediately that every 
cDG 
is collider equivalent with itself.

We do not need to consider walks in the above definitions (only paths) as we 
assume that all loops are included and therefore all nodes are 
collider connected to themselves by assumption. If there is a collider walk 
between $\alpha$ and $\beta$ ($\alpha\neq\beta$), then there is also a 
collider path. Furthermore, if a collider walk between $\alpha$ and $\beta$ 
($\alpha\neq \beta$) in $\D_1$ is covered by a collider walk in $\D_2$, 
then it is also covered by a collider path, and we see that one would obtain an 
equivalent definition by using collider walks instead of collider paths in 
Definitions \ref{def:collCover} and \ref{def:collEq}.

\begin{rem}
	Collider equivalence implies that two graphs have the same weak inducing 
	paths in the following sense. Assume $\omega$ is a weak inducing path 
	between 
	$\alpha$ 
	and $\beta$ in $\D_1$, and that $\D_1$ and $\D_2$ are collider equivalent 
	and have the same directed edges. 
	In $\D_2$, there exists a collider path, 
	$\bar{\omega}$, such that every nonendpoint node is an ancestor of a node 
	on $\omega$, i.e., an ancestor of $\{\alpha,\beta\}$ using the fact that 
	$\omega$ 
	is a weak inducing path. This means that $\bar{\omega}$ is a weak inducing 
	path in $\D_2$.
\end{rem}

\begin{lem}
	Let $\D_1=(V,E_1)$, $\D_2=(V,E_2)$ be cDGs that contain every loop. If 
	$\D_1$ and $\D_2$ are not collider equivalent, then 
	they are not Markov 
	equivalent.
	\label{lem:notCEthenNotME}
\end{lem}

\begin{proof}
	Assume that $\D_1$ and $\D_2$ are not collider equivalent. If $\D_1$ and 
	$\D_2$ do not have the same directed edges, then they are not Markov 
	equivalent (Corollary \ref{cor:IDunCorr}), and we can therefore assume that 
	the directed edges are the same. Assume 
	that there exists $\alpha,\beta\in V$ such that there is a collider path 
	between $\alpha$ and $\beta$ in $\D_2$,
	
	$$
	\alpha \sim \bar{\gamma}_1 \sim \ldots \sim \bar{\gamma}_k \sim \beta
	$$
	
	\noindent which is not covered in $\D_1$ (both 
	graphs 
	contain all loops, so $\alpha\neq \beta$). This means that on every 
	collider 
	path between $\alpha$ and $\beta$ in $\D_1$, there exists a collider 
	$\gamma$ such that $\gamma\notin \an(\alpha,\beta)$ and $\gamma\notin 
	\cup_j 
	\an(\bar{\gamma}_j)$. Now consider the set $D = \an(\alpha,\beta) \cup 
	\left[\cup_j 
	\an(\bar{\gamma}_j)\right] \setminus \{\alpha,\beta\}$. Note that $\beta$ 
	is not 
	$\mu$-separated from $\alpha$ given $D$ in $\D_2$ as 
	$\beta\rightarrow_{\D_2}\beta$, and we will argue that 
	$\beta$ is $\mu$-separated from $\alpha$ given $D$ in $\D_1$ 
	showing that these graphs are not Markov equivalent. Consider any walk 
	between $\alpha$ and $\beta$ in $\D_1$. It suffices to consider a path, 
	$\pi$,  
	between $\alpha$ and $\beta$ composed 
	with the edge 
	$\beta\rightarrow\beta$ (as $\beta\notin D$). Assume first that $\pi$ is a 
	collider 
	path. If it is open, then every 
	nonendpoint node is an 
	ancestor of 
	$\alpha$, $\beta$, or $\bar{\gamma}_j$ for some $j$, which is a 
	contradiction. Assume 
	instead that there exists a noncollider (different from $\alpha$ and 
	$\beta$) on the path. There must also exist 
	a collider (otherwise $\pi$ is closed), and the collider is a descendant of 
	the noncollider. The collider is either closed, or it is an ancestor of 
	either $\{\alpha,\beta\}$ or of $\cup_i \bar{\gamma}_i$. In the latter 
	case, the path is closed in the noncollider.
\end{proof}

\begin{figure}
	\centering
	\begin{minipage}{0.3\linewidth}
		\begin{tikzpicture}
		\node[shape=circle,draw=black] (a) at (0,2) {$\alpha$};
		\node[shape=circle,draw=black] (b) at (2,2) {$\beta$};
		\node[shape=circle,draw=black] (c) at (0,0) {$\gamma$};
		\node[shape=circle,draw=black] (d) at (2,0) {$\delta$};
		
		\path [->] (a) edge [bend left = 10] node {} (b);
		\path [->] (b) edge [bend left = 10] node {} (d);
		\path [<-] (d) edge [bend left = 10] node {} (c);
		\path [->] (a) edge [bend right = 10] node {} (c);
		\path [|-|](b) edge [bend right = 10] node {} (d);
		\path [|-|](d) edge [bend right = 10] node {} (c);
		\draw[every loop/.append style={->}] (b) edge[loop above] 
		node {}  (b);
		\draw[every loop/.append style={->}] (a) edge[loop above] 
		node {}  (a);
		\draw[every loop/.append style={->}] (c) edge[loop below] 
		node {}  (c);
		\draw[every loop/.append style={->}] (d) edge[loop below] 
		node {}  (d);	
		\end{tikzpicture}
	\end{minipage}
	\begin{minipage}{0.3\linewidth}
		\begin{tikzpicture}
		\node[shape=circle,draw=black] (a) at (0,2) {$\alpha$};
		\node[shape=circle,draw=black] (b) at (2,2) {$\beta$};
		\node[shape=circle,draw=black] (c) at (0,0) {$\gamma$};
		\node[shape=circle,draw=black] (d) at (2,0) {$\delta$};
		
		\path [->] (a) edge [bend left = 10] node {} (b);
		\path [->] (b) edge [bend left = 10] node {} (d);
		\path [<-] (d) edge [bend left = 10] node {} (c);
		\path [->] (a) edge [bend right = 10] node {} (c);
		\path [|-|](b) edge [bend right = 10] node {} (d);
		\draw[every loop/.append style={->}] (b) edge[loop above] 
		node {}  (b);
		\draw[every loop/.append style={->}] (a) edge[loop above] 
		node {}  (a);
		\draw[every loop/.append style={->}] (c) edge[loop below] 
		node {}  (c);
		\draw[every loop/.append style={->}] (d) edge[loop below] 
		node {}  (d);	
		\end{tikzpicture}
	\end{minipage}
	\begin{minipage}{0.3\linewidth}
		\begin{tikzpicture}
		\node[shape=circle,draw=black] (a) at (0,2) {$\alpha$};
		\node[shape=circle,draw=black] (b) at (2,2) {$\beta$};
		\node[shape=circle,draw=black] (c) at (0,0) {$\gamma$};
		\node[shape=circle,draw=black] (d) at (2,0) {$\delta$};
		
		\path [->] (a) edge [bend left = 10] node {} (b);
		\path [->] (b) edge [bend left = 10] node {} (d);
		\path [<-] (d) edge [bend left = 10] node {} (c);
		\path [->] (a) edge [bend right = 10] node {} (c);
		\path [|-|](d) edge [bend right = 10] node {} (c);
		\draw[every loop/.append style={->}] (b) edge[loop above] 
		node {}  (b);
		\draw[every loop/.append style={->}] (a) edge[loop above] 
		node {}  (a);
		\draw[every loop/.append style={->}] (c) edge[loop below] 
		node {}  (c);
		\draw[every loop/.append style={->}] (d) edge[loop below] 
		node {}  (d);	
		\end{tikzpicture}
	\end{minipage}
	\begin{minipage}{0.3\linewidth}
		\begin{tikzpicture}
		\node[shape=circle,draw=black] (a) at (0,2) {$\alpha$};
		\node[shape=circle,draw=black] (b) at (2,2) {$\beta$};
		\node[shape=circle,draw=black] (c) at (0,0) {$\gamma$};
		\node[shape=circle,draw=black] (d) at (2,0) {$\delta$};
		
		\path [|-|] (a) edge [bend left = 10] node {} (b);
		\path [|-|] (b) edge [bend left = 10] node {} (d);
		\path [|-|] (d) edge [bend left = 10] node {} (c);
		\path [|-|] (a) edge [bend right = 10] node {} (c);
		\path [->](b) edge [bend right = 10] node {} (c);
		\path [<-](b) edge [bend left = 10] node {} (c);
		\draw[every loop/.append style={->}] (b) edge[loop above] 
		node {}  (b);
		\draw[every loop/.append style={->}] (a) edge[loop above] 
		node {}  (a);
		\draw[every loop/.append style={->}] (c) edge[loop below] 
		node {}  (c);
		\draw[every loop/.append style={->}] (d) edge[loop below] 
		node {}  (d);	
		\end{tikzpicture}
	\end{minipage}
	\begin{minipage}{0.3\linewidth}
		\begin{tikzpicture}
		\node[shape=circle,draw=black] (a) at (0,2) {$\alpha$};
		\node[shape=circle,draw=black] (b) at (2,2) {$\beta$};
		\node[shape=circle,draw=black] (c) at (0,0) {$\gamma$};
		\node[shape=circle,draw=black] (d) at (2,0) {$\delta$};
		
		\path [|-|] (a) edge [bend left = 10] node {} (b);
		\path [|-|] (d) edge [bend left = 10] node {} (c);
		\path [|-|] (a) edge [bend right = 10] node {} (c);
		\path [->](b) edge [bend right = 10] node {} (c);
		\path [<-](b) edge [bend left = 10] node {} (c);
		\draw[every loop/.append style={->}] (b) edge[loop above] 
		node {}  (b);
		\draw[every loop/.append style={->}] (a) edge[loop above] 
		node {}  (a);
		\draw[every loop/.append style={->}] (c) edge[loop below] 
		node {}  (c);
		\draw[every loop/.append style={->}] (d) edge[loop below] 
		node {}  (d);	
		\end{tikzpicture}
	\end{minipage}
	\begin{minipage}{0.3\linewidth}
		\begin{tikzpicture}
		\node[shape=circle,draw=black] (a) at (0,2) {$\alpha$};
		\node[shape=circle,draw=black] (b) at (2,2) {$\beta$};
		\node[shape=circle,draw=black] (c) at (0,0) {$\gamma$};
		\node[shape=circle,draw=black] (d) at (2,0) {$\delta$};
		
		\path [|-|] (d) edge [bend left = 10] node {} (c);
		\path [|-|] (a) edge [bend right = 10] node {} (c);
		\path [->](b) edge [bend right = 10] node {} (c);
		\path [<-](b) edge [bend left = 10] node {} (c);
		\draw[every loop/.append style={->}] (b) edge[loop above] 
		node {}  (b);
		\draw[every loop/.append style={->}] (a) edge[loop above] 
		node {}  (a);
		\draw[every loop/.append style={->}] (c) edge[loop below] 
		node {}  (c);
		\draw[every loop/.append style={->}] (d) edge[loop below] 
		node {}  (d);	
		\end{tikzpicture}
	\end{minipage}
	\begin{minipage}{0.3\linewidth}
		\begin{tikzpicture}
		\node[shape=circle,draw=black] (a) at (0,2) {$\alpha$};
		\node[shape=circle,draw=black] (b) at (2,2) {$\beta$};
		\node[shape=circle,draw=black] (c) at (0,0) {$\gamma$};
		\node[shape=circle,draw=black] (d) at (2,0) {$\delta$};
		
		\path [|-|] (a) edge [bend left = 10] node {} (b);
		\path [|-|] (b) edge [bend left = 10] node {} (d);
		\path [|-|] (d) edge [bend left = 10] node {} (c);
		\path [|-|] (a) edge [bend right = 10] node {} (c);
		\path [->](b) edge [bend right = 10] node {} (c);
		\draw[every loop/.append style={->}] (b) edge[loop above] 
		node {}  (b);
		\draw[every loop/.append style={->}] (a) edge[loop above] 
		node {}  (a);
		\draw[every loop/.append style={->}] (c) edge[loop below] 
		node {}  (c);
		\draw[every loop/.append style={->}] (d) edge[loop below] 
		node {}  (d);	
		\end{tikzpicture}
	\end{minipage}
	\begin{minipage}{0.3\linewidth}
		\begin{tikzpicture}
		\node[shape=circle,draw=black] (a) at (0,2) {$\alpha$};
		\node[shape=circle,draw=black] (b) at (2,2) {$\beta$};
		\node[shape=circle,draw=black] (c) at (0,0) {$\gamma$};
		\node[shape=circle,draw=black] (d) at (2,0) {$\delta$};
		
		\path [|-|] (a) edge [bend left = 10] node {} (b);
		\path [|-|] (d) edge [bend left = 10] node {} (c);
		\path [|-|] (a) edge [bend right = 10] node {} (c);
		\path [->](b) edge [bend right = 10] node {} (c);
		\draw[every loop/.append style={->}] (b) edge[loop above] 
		node {}  (b);
		\draw[every loop/.append style={->}] (a) edge[loop above] 
		node {}  (a);
		\draw[every loop/.append style={->}] (c) edge[loop below] 
		node {}  (c);
		\draw[every loop/.append style={->}] (d) edge[loop below] 
		node {}  (d);	
		\end{tikzpicture}
	\end{minipage}
	\begin{minipage}{0.3\linewidth}
		\begin{tikzpicture}
		\node[shape=circle,draw=black] (a) at (0,2) {$\alpha$};
		\node[shape=circle,draw=black] (b) at (2,2) {$\beta$};
		\node[shape=circle,draw=black] (c) at (0,0) {$\gamma$};
		\node[shape=circle,draw=black] (d) at (2,0) {$\delta$};
		
		\path [|-|] (d) edge [bend left = 10] node {} (c);
		\path [->](b) edge [bend right = 10] node {} (c);
		\path [|-|](b) edge [bend right = 10] node {} (a);
		\path [|-|](b) edge [bend left = 10] node {} (d);
		\draw[every loop/.append style={->}] (b) edge[loop above] 
		node {}  (b);
		\draw[every loop/.append style={->}] (a) edge[loop above] 
		node {}  (a);
		\draw[every loop/.append style={->}] (c) edge[loop below] 
		node {}  (c);
		\draw[every loop/.append style={->}] (d) edge[loop below] 
		node {}  (d);	
		\end{tikzpicture}
	\end{minipage}
	\caption{Markov equivalence in cDGs. First row: these are three 
		members of a 
		Markov equivalence class of size 21. The only restriction on  $2^5$ 
		combinations of blunt edges (all but $\beta\unEdge\gamma$ can be 
		present) is the fact
		that we cannot have both $\alpha \unEdge \beta$ and $\alpha \unEdge 
		\gamma$ present and that either $(\alpha, \delta)$, $(\beta,\delta)$, 
		or 
		$(\gamma,\delta)$ are spouses as otherwise there would not be a weak 
		inducing 
		path between 
		$\alpha$ and $\delta$. Second row: 
		these graphs are Markov equivalent. 
		The collider path $\alpha \unEdge \beta \unEdge \delta$ in the first 
		graph 
		is 
		`covered'  in the two others by the walk $\alpha \unEdge \gamma \unEdge 
		\delta$ as 
		$\gamma\in\an(\beta)$. The edge $\beta\unEdge\delta$ is `covered' by 
		the 
		inducing 
		path $\delta \unEdge \gamma \leftarrow \beta$ in the center and right 
		graphs 
		of 
		the row. The equivalence class of these graphs has cardinality 16 which 
		is 
		every combination of blunt edges ($2^5$, excluding 
		$\alpha\unEdge\delta$ which cannot be in a Markov equivalent graph) 
		that makes the graph 
		connected via blunt edges. Third row: the first graph is not collider 
		equivalent with 
		the 
		following two: the collider path $\alpha \unEdge \beta \unEdge \delta$ 
		is 
		not covered 
		by any collider path in the second graph. The collider path $\alpha 
		\unEdge 
		\gamma$ is not covered by any collider path in the third.}
	\label{fig:meExmps1}
\end{figure}
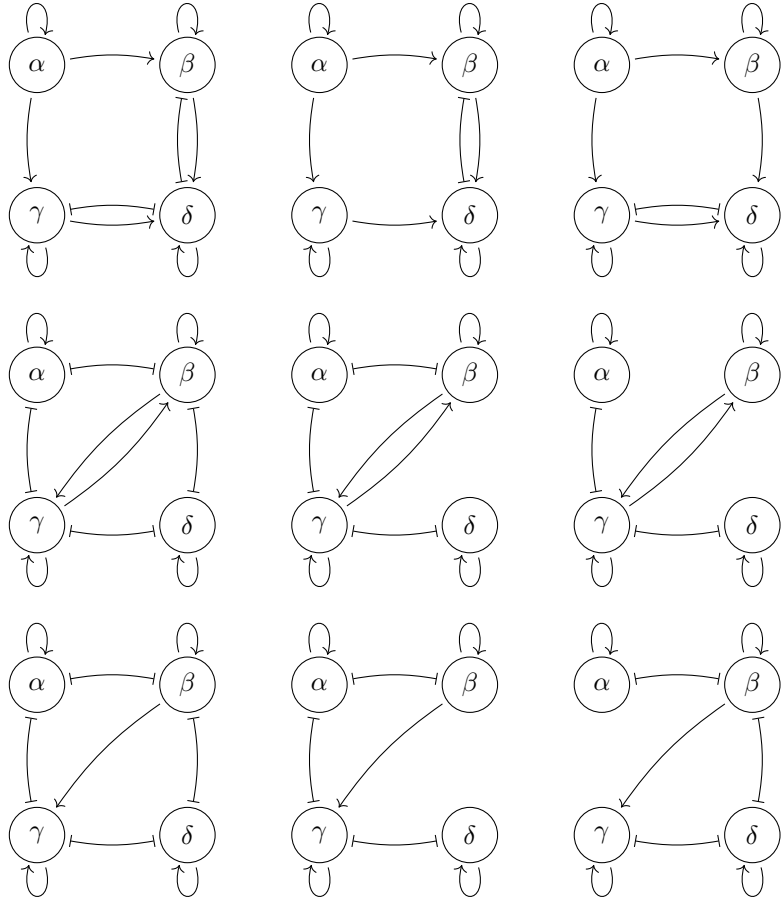

\begin{prop}
	Assume $\alpha,\beta\notin C$. If $\omega$ is a collider path between 
	$\alpha$ 
	and 
	$\beta$ such that every 
	collider is in $\an(\{\alpha,\beta\}\cup C)$, 
	then there is a walk between $\alpha$ and $\beta$ such that no noncollider 
	is in $C$ and every collider is in $\an(C)$.
	\label{prop:mConnChangeAn}
\end{prop}

A more general result was shown by \citet{richardson2003} in the 
case of $m$-separation (Definition \ref{def:mSep} in Appendix 
\ref{app:proofMarkov}) in directed mixed graphs using a similar proof.

\begin{proof}
	In the original graph, $\D$, we add directed edges such that every node 
	in $C$ is a parent of $\alpha$. Now the path is a weak inducing path, in 
	this 
	larger graph $\D^+$. Using Proposition \ref{prop:inSepWIP}, we can find a 
	$\mu$-connecting walk from $\alpha$ to $\beta$ given $C$ in $\D^+$, and 
	therefore a walk between $\alpha$ and $\beta$ such that every noncollider 
	is 
	not in $C$ and every collider is in $\an(C)$. This walk is also in $\D$ as 
	it 
	cannot contain an edge with a tail at $\gamma\in C$. In $\D$, we see that 
	every collider is still in $\an(C)$ and the result follows.
\end{proof}

\begin{thm}[Markov equivalence of cDGs]
	Let $\D_1 = (V,E_1)$ and $\D_2 = (V,E_2)$ be cDGs that contain every loop. 
	The 
	graphs 
	$\D_1$ and $\D_2$ are Markov equivalent if and only if they have the same 
	directed edges and are 
	collider equivalent.
	\label{thm:meChar}
\end{thm}

We give a direct proof of this theorem. One can also use the augmentation 
criterion (Appendix \ref{app:aug}) to show this result.

\begin{proof}
	Assume first that $\D_1$ and $\D_2$ have the same directed edges and are 
	collider equivalent. Then $\an_{\D_1}(C) = \an_{\D_2}(C)$ for all 
	$C\subseteq V$ so we will omit the subscript and write simply $\an(C)$. Let 
	$\omega$ 
	denote a $\mu$-connecting walk from $\alpha$ to $\beta$ given 
	$C$ in $\D_1$. We will argue that we can also find a $\mu$-connecting walk 
	in $\D_2$. We say that a nontrivial subwalk of $\omega$ is a {\it maximal 
	collider 
	segment} if all its nonendpoint nodes are colliders on $\omega$, its 
	endpoint nodes are not colliders, and it contains at least one blunt 
	edge 
	(note that on a general walk this should 
	be read as {\it instances} of these nodes and edges as nodes and edges may 
	be repeated on a 
	walk). We can partition $\omega$ into a sequence of subwalks such that 
	every subwalk is either a maximal collider segment, or a subwalk consisting 
	of directed edges only. We note that maximal collider segments may be 
	adjacent, i.e., share an endpoint. Every segment of $\omega$
	that consists of directed edges only is also present in $\D_2$. Consider a 
	maximal collider segment. This is necessarily a collider walk in $\D_1$. 
	Then there 
	exists a collider path in $\D_1$, and therefore a covering collider path, 
	$\rho$, in 
	$\D_2$ using 
	collider equivalence. Assume that $\rho$ is 
	between $\delta$ and $\varepsilon$. 
	$\delta $ and $\varepsilon$ are noncolliders on $\omega$, or 
	endpoint nodes on $\omega$. If $\delta = \alpha$ or $\varepsilon = \alpha$, 
	then they are not in $C$. The final 
	edge must be directed and point towards $\beta$ and therefore it is not in 
	a maximal collider segment, and $\delta$ and $\varepsilon$ are not the 
	final node on $\omega$. In either case, we see that 
	$\delta,\varepsilon\notin C$. We 
	will now find an open (given 
	$C$) walk between $\delta$ and $\varepsilon$ using $\rho$. We know that 
	$\rho$ is a collider path and that every nonendpoint node on $\rho$ is an 
	ancestor of $\{\alpha,\beta\}$ or of a collider in the original maximal 
	collider segment, and therefore to $C$. It follows from Proposition 
	\ref{prop:mConnChangeAn} that we can 
	find a walk between $\delta$ and $\varepsilon$ such that no noncollider is 
	in $C$ and every collider is in $\an(C)$. We create a 
	walk from $\alpha$ to $\beta$ in $\D_2$ by simply substituting each maximal 
	collider segment with the corresponding open walk. This walk is open in any 
	node which is not an endpoint of a maximal collider segment. If an endpoint 
	of 
	maximal collider node changes collider status on this new walk, then it 
	must be a noncollider on $\omega$ and a 
	parent of a node in $\an(C)$, i.e., also in $\an(C)$ itself. Finally, 
	we note that the last segment (into $\beta$) is not a maximal collider 
	segment and therefore still has a head into $\beta$.
	
	On the other hand, if they do not have the same directed edges, it follows 
	from Proposition \ref{prop:dirEdgeSep} that they are not Markov equivalent. 
	If they are not collider equivalent, it follows from Lemma 
	\ref{lem:notCEthenNotME} that they are not Markov equivalent.
\end{proof}

\begin{figure}
	\centering
	\begin{minipage}{0.45\linewidth}
		\centering
		\begin{tikzpicture}
		\node[shape=circle,draw=black] (a) at (0,2) {$\alpha$};
		\node[shape=circle,draw=black] (b) at (2,2) {$\beta$};
		\node[shape=circle,draw=black] (c) at (4,2) {$\gamma$};
		\node[shape=circle,draw=black] (d) at (1,0) {$\delta$};
		\node[shape=circle,draw=black] (e) at (3,0) {$\varepsilon$};
		
		\path [|-|] (a) edge [bend left = 10] node {} (b);
		\path [|-|] (d) edge [bend left = 0] node {} (e);
		\path [|-|] (a) edge [bend right = 10] node {} (d);
		\path [<-] (d) edge [bend left = 20] node {} (e);
		\path [->] (d) edge [bend right = 20] node {} (e);
		\path [|-|](b) edge [bend left = 10] node {} (c);
		\path [|-|](e) edge [bend right = 10] node {} (c);
		\draw[every loop/.append style={->}] (b) edge[loop above] 
		node {}  (b);
		\draw[every loop/.append style={->}] (a) edge[loop above] 
		node {}  (a);
		\draw[every loop/.append style={->}] (c) edge[loop above] 
		node {}  (c);
		\draw[every loop/.append style={->}] (d) edge[loop below] 
		node {}  (d);
		\draw[every loop/.append style={->}] (e) edge[loop below] 
		node {}  (e);	
		\end{tikzpicture}
	\end{minipage}
	\begin{minipage}{0.45\linewidth}
		\centering
		\begin{tikzpicture}
		\node[shape=circle,draw=black] (a) at (0,2) {$\alpha$};
		\node[shape=circle,draw=black] (b) at (2,2) {$\beta$};
		\node[shape=circle,draw=black] (c) at (4,2) {$\gamma$};
		\node[shape=circle,draw=black] (d) at (1,0) {$\delta$};
		\node[shape=circle,draw=black] (e) at (3,0) {$\varepsilon$};
		
		\path [|-|] (a) edge [bend left = 10] node {} (b);
		\path [|-|] (a) edge [bend right = 10] node {} (d);
		\path [<-] (d) edge [bend left = 10] node {} (e);
		\path [->] (d) edge [bend right = 10] node {} (e);
		\path [|-|](b) edge [bend left = 10] node {} (c);
		\path [|-|](e) edge [bend right = 10] node {} (c);
		\draw[every loop/.append style={->}] (b) edge[loop above] 
		node {}  (b);
		\draw[every loop/.append style={->}] (a) edge[loop above] 
		node {}  (a);
		\draw[every loop/.append style={->}] (c) edge[loop above] 
		node {}  (c);
		\draw[every loop/.append style={->}] (d) edge[loop below] 
		node {}  (d);
		\draw[every loop/.append style={->}] (e) edge[loop below] 
		node {}  (e);	
		\end{tikzpicture}
		\end{minipage}
	\caption{Markov equivalence in cDGs. The two graphs 
	have the same weak inducing paths, but 
	are not Markov equivalent as the collider path $\alpha \unEdge \delta 
	\unEdge \epsilon 
	\unEdge \gamma$ is not covered in the right graph. 
	}
	\label{fig:meExmps2}
\end{figure}
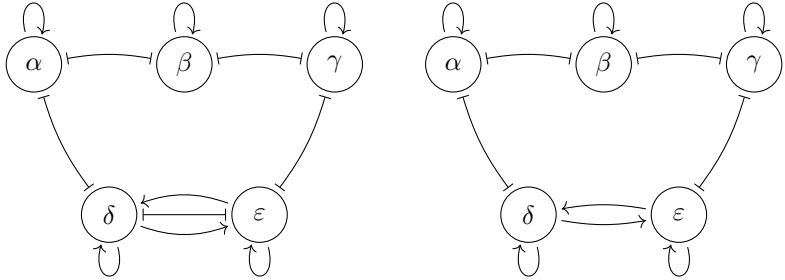

In the case of {\it directed acyclic graphs} it holds 
that Markov equivalent graphs have the same adjacencies, however, this is not 
true in the case of cDGs, and in fact, it is also not true among maximal cDGs 
(Definition \ref{def:max}) as seen in Figure \ref{fig:maxDiffAdj}.

\begin{prop}
	Let $\D= (V,E)$ be a cDG, and let $\alpha,\beta \in V$. Let $e$ 
	denote the blunt edge between $\alpha$ and $\beta$. If $\alpha$ and 
	$\beta$ are connected by a weak inducing path consisting of blunt 
	edges only, then $\D 
	+ e\in [\D]$. 
\end{prop}

\begin{proof}
	Let $\omega$ be a $\mu$-connecting walk between $\delta$ and $\varepsilon$ 
	given $C$ 
	in $\D + e$. 
	If $e$ is not on $\omega$, then $\omega$ is also in $\D$ and 
	connecting as the ancestral relations are the same in $\D$ and $\D + e$. 
	If $e$ is on $\omega$, then consider the weak inducing path 
	between $\alpha$ and 
	$\beta$ in $\D$ that consists of blunt edges only. Using a proof similar to 
	that of 
	Proposition 
	\ref{prop:inSepWIP} (let $k$ in the proof of that proposition fulfil the 
	additional assumptions that the corresponding walk in that proof has necks 
	at both 
	endpoints, only contains one instance of $\alpha$, and does not contain any 
	instances of $\beta$), one 
	can show that there exists an open walk between 
	$\alpha$ and $\beta$ given $C\setminus \{\alpha,\beta\}$ in $\D$ which has 
	necks at both ends and which only contains one instance of both $\alpha$ 
	and 
	$\beta$. This means 
	that replacing $\alpha\unEdge\beta$ with this walk gives a $\mu$-connecting 
	walk given $C$ in $\D$. 
\end{proof}

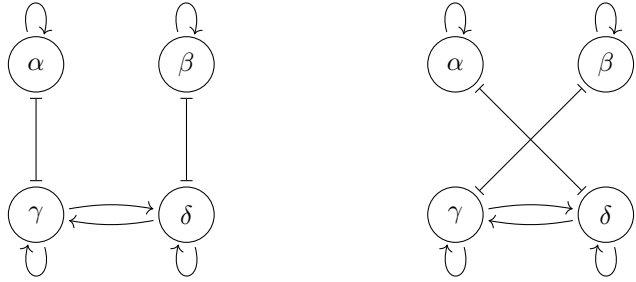
\begin{figure}
	\centering
	\begin{minipage}{0.45\linewidth}
		\centering
		\begin{tikzpicture}
		\node[shape=circle,draw=black] (a) at (0,2) {$\alpha$};
		\node[shape=circle,draw=black] (b) at (2,2) {$\beta$};
		\node[shape=circle,draw=black] (c) at (0,0) {$\gamma$};
		\node[shape=circle,draw=black] (d) at (2,0) {$\delta$};
		
		\path [|-|] (a) edge [bend right = 0] node {} (c);
		\path [|-|](b) edge [bend right = 0] node {} (d);
		\path [->](d) edge [bend left = 10] node {} (c);
		\path [<-](d) edge [bend left = -10] node {} (c);
		\draw[every loop/.append style={->}] (b) edge[loop above] 
		node {}  (b);
		\draw[every loop/.append style={->}] (a) edge[loop above] 
		node {}  (a);
		\draw[every loop/.append style={->}] (c) edge[loop below] 
		node {}  (c);
		\draw[every loop/.append style={->}] (d) edge[loop below] 
		node {}  (d);	
		\end{tikzpicture}
	\end{minipage}
	\begin{minipage}{0.45\linewidth}
		\centering
		\begin{tikzpicture}
		\node[shape=circle,draw=black] (a) at (0,2) {$\alpha$};
		\node[shape=circle,draw=black] (b) at (2,2) {$\beta$};
		\node[shape=circle,draw=black] (c) at (0,0) {$\gamma$};
		\node[shape=circle,draw=black] (d) at (2,0) {$\delta$};
		
		\path [|-|] (a) edge [bend right = 0] node {} (d);
		\path [|-|](b) edge [bend right = 0] node {} (c);
		\path [->](d) edge [bend left = 10] node {} (c);
		\path [<-](d) edge [bend left = -10] node {} (c);
		\draw[every loop/.append style={->}] (b) edge[loop above] 
		node {}  (b);
		\draw[every loop/.append style={->}] (a) edge[loop above] 
		node {}  (a);
		\draw[every loop/.append style={->}] (c) edge[loop below] 
		node {}  (c);
		\draw[every loop/.append style={->}] (d) edge[loop below] 
		node {}  (d);	
		\end{tikzpicture}
	\end{minipage}
	\caption{The two cDGs constitute a Markov equivalence class, and they are 
	both seen to be maximal. However, they do not have the same adjacencies. A 
	similar phenomenon can occur in DGs (without loops) under $d$-separation
	\citep{richardson1996, richardson1997}.}
	\label{fig:maxDiffAdj}
\end{figure}

\subsection{Markov equivalent permutation of nodes}

The example in Figure \ref{fig:maxDiffAdj} shows a characteristic 
of some Markov equivalent cDGs. In the example, one can obtain one graph from 
the other by a permutation of the endpoints of blunt edges within the set 
$\{\gamma,\delta\}$. In this section, we formulate sufficient conditions for a 
cDG to be Markov equivalent with a {\it permutation graph}.

\begin{defn}[Cyclic set]
	We say that $S \subseteq V$ is a \emph{cyclic set} if for every 
	$(\alpha,\beta)\in S\times S$, it holds that $\alpha \in \an(\beta)$.
	\label{def:cyclicSet}
\end{defn}

The following is a formal definition of a {\it permutation graph} as 
illustrated in the example of Figure \ref{fig:maxDiffAdj}.

\begin{defn}[Permutation graph]
	Let $\D= (V,E)$ be a cDG and let $\rho$ be a 
	permutation of 
	the node set $V$. We 
	define
	$\mathcal{P}_\rho(\D)$ as the cDG on nodes $V$ such that
	
	\begin{alignat}{3}
		\alpha & \rightarrow_{\mathcal{P}_\rho(\D)}\beta & \text{ if and only if
		} &\alpha \rightarrow_\D\beta, \\
	 \rho(\alpha) & \unEdge_{\mathcal{P}_\rho(\D)}\rho(\beta) & \text{ if and 
	 only if } 	
	 &\alpha \unEdge_\D\beta.
	\end{alignat}
\end{defn}

\begin{prop}
	Let $\D = (V,E)$ be a cDG which contains all loops and let $S\subseteq V$. 
	Let $\rho$ be 
	a 
	permutation of $V$ such that $\rho(\alpha) = \alpha$ for all $\alpha\notin 
	S$. If $\beta \rightarrow_\D \gamma$ and $\pa(\beta)=\pa(\gamma)$ for all 
	$\beta,\gamma\in S$, then 	
	$\mathcal{P}_\rho(\D) 
	\in [\D]$. 
	\label{prop:perm}
\end{prop}

Note that the condition that $\beta\rightarrow_\D\gamma$ for all 
$\beta,\gamma\in 
S$ 
implies that $S$ is a cyclic set.

\begin{proof}
	The graphs $\D$ and $\mathcal{P}_\rho(\D) $ have the same directed edges so 
	it suffices to show that they are collider equivalent (Theorem 
	\ref{thm:meChar}). 
	Any permutation can be written as a composition of transpositions so it 
	suffices to prove the result for a permutation, $\rho$, such that 
	$\rho(\alpha) = \beta$, $\rho(\beta) = \alpha$, and $\rho(\gamma) = \gamma$ 
	for all $\gamma\neq \alpha,\beta$. Let $\pi$ be a collider path in $\D$,
	
	$$
	\gamma \sim \delta_1 \sim \ldots \sim \delta_k \sim \varepsilon.
	$$
	
	\noindent If $\gamma,\varepsilon\notin \{\alpha,\beta\}$, then the path
	
	$$
	\gamma \sim \rho(\delta_1) \sim \ldots \sim \rho(\delta_k) \sim \varepsilon
	$$

	\noindent is in the permutation graph and is covering, using that $\alpha$ 
	and 
	$\beta$ have the same parent set. If, e.g., $\gamma=\alpha \unEdge 
	\delta_1$ 
	on the original path, then we can substitute this for $\alpha \rightarrow 
	\beta \unEdge \delta_1$ to obtain a covering walk in the permutation graph.
	Similar arguments in each case show that any collider path in 
	$\D$ is covered in the permutation graph. Repeating the above argument 
	starting from the permutation graph and using the transposition $\rho^{-1} 
	= \rho$ shows that 
	the 
	two 
	graphs are Markov equivalent. 
\end{proof}

Figure \ref{fig:maxDiffAdj} shows two graphs that are Markov equivalent by 
Proposition \ref{prop:perm}. In some graphs one can find permutations, 
not 
fulfilling the 
assumptions of Proposition 
\ref{prop:perm}, that generate Markov equivalent graphs, and this proposition 
is 
therefore not a necessary condition for Markov equivalence under 
permutation of blunt edges. One example is in the first row of Figure 
\ref{fig:meExmps1}. The center and right graphs are Markov equivalent and one 
is generated from the other by permuting the blunt edges of $\beta$ and 
$\gamma$, however, the 
conditions of Proposition \ref{prop:perm} are not fulfilled.

\section{Deciding Markov equivalence}
\label{sec:decME}

In this section, we will consider the problem of deciding Markov equivalence 
algorithmically. That is, given two cDGs on the same node set, how can we 
decide 
if they are Markov equivalent or not? A possible starting point is Theorem 
\ref{thm:meChar}. While it is computationally easy to check whether the 
directed edges of two cDGs are the same (quadratic in the number of nodes in 
their 
mutual node set), collider equivalence could be hard as there may be 
exponentially 
many collider paths in a cDG. In this section, we give a different 
characterization of Markov equivalence (Theorem \ref{thm:ancME}) which proves 
the correctness of a simple 
algorithm (Algorithm \ref{alg:ancME}) for deciding Markov equivalence of two 
cDGs. This algorithm avoids 
checking each collider path explicitly. However, in the worst case it also has 
a superpolynomial runtime which is to be expected due to the complexity result 
in Theorem \ref{thm:meCoNP}.

The {\it 
directed part} of a cDG, $\mathbb{D}(\D) = (V,F)$, is the DG on nodes $V$ such 
that $\alpha \rightarrow_{\mathbb{D}(\D)} \beta$ if and only if $\alpha 
\rightarrow_\D \beta$. The {\it blunt part} of a cDG, $\mathbb{U}(\D)$, is the 
cDG obtained by removing all directed edges. The {\it blunt components} of $\D$ 
are the connected components of $\mathbb{U}(\D)$. We say that $\D_1 = (V,E_1)$ 
and $\D_2 = (V,E_2)$ have {\it the same collider 
connections} if it holds for all $\alpha \in V$ and $\beta\in V$ that $\alpha$ 
and 
$\beta$ are collider connected in $\D_1$ if and only if they are 
collider connected in $\D_2$. We say that a subset of nodes, $A$, is {\it 
ancestral} if $A = \an(A)$. We will throughout only consider cDGs that contain 
every loop.

We start from the following result which is seen to be a reformulation of the 
augmentation criterion (Appendix \ref{app:aug}).

\begin{thm}
	Let $\D_1 = (V,E_1)$ and $\D_2 = (V,E_2)$ be cDGs (both containing all 
	loops) such that 
	$\mathbb{D}(\D_1) = \mathbb{D}(\D_2)$. $\D_1$ and $\D_2$ are Markov 
	equivalent if and 
	only if for every ancestral set, it holds 
	that 
	$(\D_1)_A$ and $(\D_2)_A$ have the same collider connections. 
	\label{thm:ancME}
\end{thm}

\begin{proof}
	Assume that there exists an ancestral set $A \subseteq V$ such that 
	$\alpha$ 
	and $\beta$ are collider connected in $(\D_1)_A$, but not in $(\D_2)_A$. 
	There exists a collider path in $\D_1$ between $\alpha$ and $\beta$. 
	Any covering path in $\D_2$ must by definition consist of nodes in $\an(A) 
	= A$ and it follows that no such path can exists. By Lemma 
	\ref{lem:notCEthenNotME}, it follows that $\D_1$ and $\D_2$ are not Markov 
	equivalent.
	
	On the other hand, assume that for every 
	ancestral set $A\subseteq V$ and every $\alpha,\beta\in A$, it holds that 
	$\alpha$ and $\beta$ are 
	collider connected in $(\D_1)_A$ if and only if $\alpha$ and $\beta$ are 
	collider connected in $(\D_2)_A$. Using Theorem \ref{thm:meChar}, it 
	suffices to show that $\D_1$ and $\D_2$ are collider equivalent. 
	Consider a collider path between $\alpha$ and $\beta$ in $\D_1$, and let 
	$C$ denote the set of nodes on this path. This path is also a collider path 
	in $(\D_1)_{\an(\{\alpha,\beta\} \cup C)}$ and by assumption we can find a 
	collider path between 
	$\alpha$ and $\beta$ in $(\D_2)_{\an(\{\alpha,\beta\} \cup C)}$ as well. This 
	collider path is in 
	$\D_2$ as well and is covering the path in $\D_1$.
\end{proof}

The above theorem can easily be turned into an algorithm for deciding if two 
cDGs 
are 
Markov equivalent (Algorithm \ref{alg:ancME}). However, 
there may be exponentially many ancestral sets in a cDG. 
For instance, in the case where the only directed edges are loops all subsets 
of $V$ are ancestral and therefore the algorithm would need to compare collider 
connections in $2^n$ pairs of graphs where $n$ is the number of nodes in the 
graphs (or $2^n-1$, omitting the empty set).

\subsection{An algorithm for deciding equivalence}

In the algorithm based on Theorem \ref{thm:ancME} we will use the {\it 
condensation} of a cDG. This is not needed, but does provide a convenient 
representation of the ancestor relations between nodes in a cyclic graph. Let 
$\D = (V,E)$ be a cDG. We say that $\alpha,\beta\in V$ are {\it strongly 
	connected} if there exists a directed path from $\alpha$ to $\beta$ 
	and a 
directed path from $\beta$ to $\alpha$, allowing trivial paths. Equivalently, 
$\alpha$ and $\beta$ are strongly connected if and only if 
$\alpha\in\an(\beta)$ and $\beta\in\an(\alpha)$. This is an 
equivalence relation on the node set of a cDG and we say that the equivalence 
classes are the {\it strongly connected components} of the graph. The 
definition of strong 
connectivity is often used in DGs \citep{introAlgo}. We simply use a 
straightforward generalization to the class of 
cDGs such that the directed part of the cDG determines strong 
connectivity. The strongly connected components are also the maximal cyclic 
sets (Definition \ref{def:cyclicSet}).

The {\it condensation} of $\D$ (also known as the {\it acyclic component graph} 
of $\D$) is the directed 
acyclic graph obtained by contracting each strongly connected component to a 
single vertex. That is, if $C_1,\ldots,C_m$ are the strongly connected 
components of $\D$ ($C_i\subseteq V$ for all $i$), 
then the condensation of $\D$ has node set $\mathbb{C} = \{C_1,\ldots,C_m\}$ 
and $C_i 
\rightarrow C_j$ if $i\neq j$ and there exists $\alpha\in C_i,\beta\in C_j$ 
such that $\alpha\rightarrow_\D\beta$ \citep{introAlgo}. We denote the 
condensation of $\D$ by $\mathcal{C}(\D)$. We also define the {\it completed 
	condensation} of $\D$, $\bar{\mathcal{C}}(\D)$, which is the graph on nodes 
$\mathbb{C}\cup \{\emptyset\}$ such that $\bar{\mathcal{C}}(\D)_\mathbb{C} = 
\mathcal{C}(\D)$ and such 
that 
$\emptyset$ is a parent of every other node and a child of none. The 
condensation and the completed condensation are both DAGs. When $\D$ has $d$ 
directed edges that are not loops, then strongly connected components can be 
found in linear time, that is, $O(n + d)$ where $n = \vert 
V\vert$ \citep{introAlgo}.

In the following, we will be considering sets of nodes in $\D$, i.e., 
subsets of $V$, as well as sets of nodes in $\mathcal{C}(\D)$, that is, subsets 
of 
$\mathbb{C}$. We write the former as capital letters, $A,B,C$. We write the 
latter as capital letters in bold font, $\mathbf{A},\mathbf{B},\mathbf{C}$, to 
emphasize that they are 
subsets of $\mathbb{C}$, not of $V$. 

\begin{prop}
	The ancestral sets in $\D$ are exactly the sets of the form $\bigcup_{C\in 
	\mathbf{A}} C$ for an  ancestral set, $\mathbf{A}$, in $\mathcal{C}(\D)$.
	\label{prop:condAnc}
\end{prop}

\begin{proof}
	Consider an ancestral set $A \subseteq V$. We can write this as a union of 
	strongly connected components, $A = \bigcup C_i$. These strongly connected 
	components must necessarily constitute an ancestral set in 
	$\mathcal{C}(\D)$.
	
	On the other hand, consider an ancestral set in $\mathcal{C}(\D)$, 
	$\mathbf{A}$, and 
	consider $\alpha \in A =\bigcup_{C\in 
		\mathbf{A}} C$. Assume that $\alpha\in C\in \mathbf{A}$. If $\beta$ is 
		an ancestor of 
	$\alpha$ in $\D$, then $\beta \in \tilde{C}$ such that $\tilde{C}$ is an 
	ancestor of $C$ in $\mathcal{C}(\D)$. By assumption, $\mathbf{A}$ is 
	ancestral, so 
	$\tilde{C}\in \mathbf{A}$ and 
	we see that $A$ is ancestral.
\end{proof}

The above proposition shows that we can consider the condensation when finding 
ancestral sets in a cDG. We let $\mathbb{A}(\D)$ denote the set of ancestral 
sets in $\D$. 
The correctness of Algorithm \ref{alg:ancME} follows from Theorem 
\ref{thm:ancME} and 
Proposition \ref{prop:condAnc}. The algorithm considers ancestral sets in 
the condensation, however, a version using ancestral sets directly in $\D_1$ is 
of course also possible. In the algorithm, one can decide collider connectivity 
by noting that $\alpha$ and $\beta$ are collider connected in a cDG, $\D$, if 
and only 
if there exists a blunt component, $\mathcal{U} = (U, E_U)$, such that 
$\alpha\in \pa_\D(U)$ and 
$\beta\in \pa_\D(U)$, using that the graphs contain every loop.

\begin{algorithm}
	\caption{Markov equivalence}
	\begin{algorithmic}
		\REQUIRE cDGs, $D_1=(V,E_1),D_2=(V,E_2)$
		\IF{$\mathbb{D}(\D_1) \neq \mathbb{D}(\D_2)$}
		 \RETURN FALSE
		\ENDIF
				\FOR{$\textbf{A} \in \mathbb{A}(\mathcal{C}(\D_1))$}
				\STATE Define $A = \bigcup_{C\in 
					\mathbf{A}} C$
				\IF{$(\D_1)_A$ and $(\D_2)_A$ do not have the same collider 
				connections}
				\RETURN FALSE
				\ENDIF
				\ENDFOR
		\RETURN TRUE
	\end{algorithmic}
	\label{alg:ancME}
\end{algorithm}

\subsection{Virtual collider tripaths}

This section describes a graphical 
structure that we will call a {\it virtual collider tripath}. We will use these 
to give a necessary condition for Markov equivalence which is computationally 
easy to check.

\begin{defn}[Virtual collider tripath]
	Let $\alpha,\beta\in V$ and let $C$ be a node in $\bar{\mathcal{C}}(\D)$, 
	i.e., $C$ 
	is a 
	strongly connected component or the empty set. We say that 
	$(\alpha,\beta,C)$ is a \emph{virtual collider tripath} if there exists a 
	(nontrivial) 
	collider path $\alpha \sim \gamma_1 \sim \ldots \gamma_m \sim \beta$ such 
	that $\gamma_i \in \an(\{\alpha,\beta\}\cup C)$ for all $i=1,\ldots,m$.
	\label{def:vct}
\end{defn}

Note that if $\alpha =\beta$, then there is no path fulfilling the requirements 
of Definition \ref{def:vct}, hence $(\alpha,\alpha,C)$ is not a virtual 
collider tripath for any $C$. \citet{richardsonPolyAlgo1996} describes {\it 
virtual adjacencies} in DGs 
equipped with 
$d$-separation. Those are structures that in terms of separation act as 
adjacencies. The idea behind virtual collider tripaths is essentially the 
same; for a fixed pair of nodes, $\alpha$ and $\beta$, a virtual collider 
tripath, $(\alpha,\beta,C)$, acts as if there exists $\gamma\in C$ such that 
$\alpha \sim \gamma \sim \beta$ is a collider walk. Note also that if $\alpha$ 
and $\beta$ are adjacent, then 
$(\alpha,\beta,C)$ is 
virtual collider tripath for any strongly connected component $C$. Finally, 
note that there 
are no restrictions on whether or not $\alpha$, $\beta$, or both are elements 
in 
the set $C \subseteq V$. 

\begin{defn}[Maximal virtual collider tripath]
	We say that a virtual collider tripath, $(\alpha,\beta,C)$, is 
	\emph{maximal} if there is no $\tilde{C}\neq C$ such that 
	$(\alpha,\beta,\tilde{C})$ is a virtual collider tripath and $\tilde{C}$ is 
	an ancestor of $C$ in $\bar{\mathcal{C}}(\D)$. 
\end{defn}

We say that two cDGs have the same (maximal) virtual collider tripaths if it 
holds that $(\alpha,\beta,C)$ is a (maximal) virtual collider tripath in $\D_1$ 
if and only if $(\alpha,\beta,C)$ is a (maximal) virtual collider tripath in 
$\D_2$.

\begin{prop}
	Let $C$ be a strongly connected component or the empty set. If 
	$(\alpha,\beta,C)$ is not a virtual collider tripath and $\alpha\neq 
	\beta$, then $\beta$ and 
	$\alpha$ are 
	$m$-separated (Definition \ref{def:mSep} in Appendix \ref{app:proofMarkov}) 
	by $\an(\{\alpha,\beta\}\cup C) \setminus 
	\{\alpha,\beta\}$.
	\label{prop:noVCTsep}
\end{prop}

\begin{proof}
	The contraposition follows from the definition of a virtual 
	collider tripath. Assume that $\omega$ is an $m$-connecting path between 
	$\alpha$ and $\beta$ given $\an(\{\alpha,\beta\}\cup C) \setminus 
	\{\alpha,\beta\}$. If it is a single edge, then $(\alpha,\beta,C)$ is a 
	virtual collider tripath for any $C$. Assume that it has length at least 
	two. If there is a noncollider, $\delta$, on $\omega$, then $\delta$ must 
	be an ancestor of $\{\alpha,\beta\}$ or of a collider. In the former case, 
	$\omega$ is closed as $\delta$ is in the conditioning set. In the latter 
	case, either $\omega$ is closed in the collider or in $\delta$. Assume 
	therefore that $\omega$ is a collider path. We see from the definition that 
	$(\alpha,\beta,C)$ is a virtual 
	collider tripath.
\end{proof}

The next theorem gives a necessary condition for Markov equivalence of cDGs.

\begin{thm}
	Let $\D_1 = (V,E_1)$ and $\D_2 = (V,E_2)$ be cDGs containing every loop. If 
	they are Markov 
	equivalent, then they have the same 
	directed edges and the same 
	maximal virtual collider tripaths.
\end{thm}

\begin{proof}
	We show this by contraposition. If $\alpha$ is a parent of $\beta$ in 
	$\D_1$, but not in $\D_2$, then it 
	follows from Corollary \ref{cor:IDunCorr} that they are not Markov 
	equivalent. Assume 
	instead that $\D_1$ and $\D_2$ have the same directed edges, and that 
	$(\alpha,\beta,C)$ 
	is a maximal virtual collider tripath in $\D_1$, but not in $\D_2$. Then it 
	follows that $\alpha\neq \beta$ as we assume all directed loops to be 
	present in both graphs. There 
	are two cases; either $(\alpha,\beta,C)$ is not a virtual collider tripath 
	in $\D_2$, or it is not maximal. In the first case, $\beta$ is 
	$\mu$-separated from $\alpha$ by $\an(\{\alpha,\beta\} \cup C)\setminus 
	\{\alpha,\beta\}$ (Proposition \ref{prop:noVCTsep}) which is seen to not be 
	the case 
	in $\D_1$. In the second case, in $\D_2$ there is a virtual collider 
	tripath $(\alpha,\beta,\tilde{C})$ such that $\tilde{C} \rightarrow C$ in 
	$\bar{\mathcal{C}}(\D_1)$ (note that $\bar{\mathcal{C}}(\D_1) = 
	\bar{\mathcal{C}}(\D_2)$) and 
	$(\alpha,\beta,\tilde{C})$ is 
	not a virtual collider tripath in $\D_1$. Repeating the above argument, we 
	see that $\D_1$ and $\D_2$ are not Markov equivalent in this case either.
\end{proof}

The example in Figure \ref{fig:mvctExmp} shows that having the same 
directed edges and the same
maximal virtual collider 
tripaths is not a sufficient condition for Markov equivalence.

	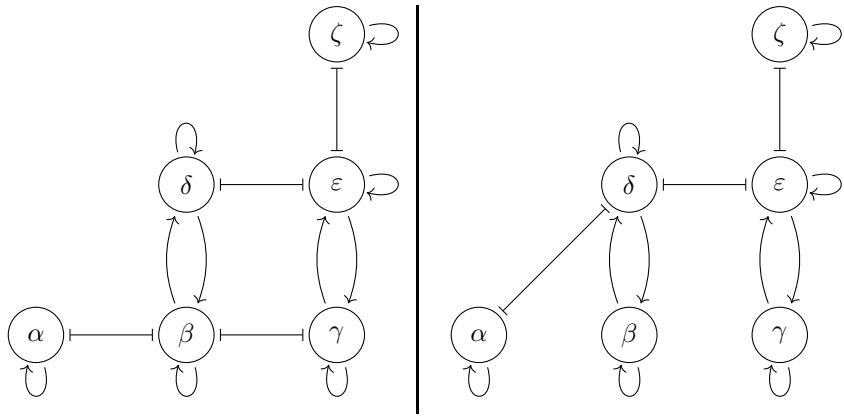
\begin{figure}
		\centering
		\begin{minipage}{.45\linewidth}
			\centering
			\begin{tikzpicture}[shorten >= 2pt, shorten <= 2pt]
			\node[shape=circle,draw=black] (a) at (0,0) {$\alpha$};
			\node[shape=circle,draw=black] (b) at (2,0) {$\beta$};
			\node[shape=circle,draw=black] (c) at (4,0) {$\gamma$};
			\node[shape=circle,draw=black] (d) at (2,2) {$\delta$};
			\node[shape=circle,draw=black] (e) at (4,2) {$\varepsilon$};
			\node[shape=circle,draw=black] (f) at (4,4) {$\zeta$};
			
			\path [->] (d) edge [bend left = 20] node {} (b);
			\path [->] (b) edge [bend left = 20] node {} (d);
			\path [->] (c) edge [bend left = 20] node {} (e);
			\path [->] (e) edge [bend left = 20] node {} (c);
			
			\path [|-|] (a) edge [bend right = 0] node {} (b);
			\path [|-|](b) edge [bend right = 0] node {} (c);
			\path [|-|] (d) edge [bend right = 0] node {} (e);
			\path [|-|](e) edge [bend right = 0] node {} (f);
			
			\draw[every loop/.append style={->}] (b) edge[loop below] 
			node {}  (b);
			\draw[every loop/.append style={->}] (a) edge[loop below] 
			node {}  (a);				
			\draw[every loop/.append style={->}] (c) edge[loop below] 
			node {}  (c);
			\draw[every loop/.append style={->}] (d) edge[loop above] 
			node {}  (d);		
			\draw[every loop/.append style={->}] (e) edge[loop right] 
			node {}  (e);
			\draw[every loop/.append style={->}] (f) edge[loop right] 
			node {}  (f);
			\end{tikzpicture}
		\end{minipage}%
		\vline \hspace{.3cm}
		\begin{minipage}{.45\linewidth}
			\centering
			\begin{tikzpicture}[shorten >= 2pt, shorten <= 2pt]
			\node[shape=circle,draw=black] (a) at (0,0) {$\alpha$};
			\node[shape=circle,draw=black] (b) at (2,0) {$\beta$};
			\node[shape=circle,draw=black] (c) at (4,0) {$\gamma$};
			\node[shape=circle,draw=black] (d) at (2,2) {$\delta$};
			\node[shape=circle,draw=black] (e) at (4,2) {$\varepsilon$};
			\node[shape=circle,draw=black] (f) at (4,4) {$\zeta$};
			
			\path [->] (d) edge [bend left = 20] node {} (b);
			\path [->] (b) edge [bend left = 20] node {} (d);
			\path [->] (c) edge [bend left = 20] node {} (e);
			\path [->] (e) edge [bend left = 20] node {} (c);
			
			\path [|-|] (a) edge [bend right = 0] node {} (d);
			\path [|-|] (d) edge [bend right = 0] node {} (e);
			\path [|-|](e) edge [bend right = 0] node {} (f);
			
			\draw[every loop/.append style={->}] (b) edge[loop below] 
			node {}  (b);
			\draw[every loop/.append style={->}] (a) edge[loop below] 
			node {}  (a);				
			\draw[every loop/.append style={->}] (c) edge[loop below] 
			node {}  (c);
			\draw[every loop/.append style={->}] (d) edge[loop above] 
			node {}  (d);		
			\draw[every loop/.append style={->}] (e) edge[loop right] 
			node {}  (e);
			\draw[every loop/.append style={->}] (f) edge[loop right] 
			node {}  (f);
			\end{tikzpicture}
		\end{minipage}
		\caption{These cDGs on nodes 
		$\{\alpha,\beta,\gamma,\delta,\varepsilon,\zeta\}$ have the same 
		maximal 
		virtual collider tripaths, however, they disagree on whether $\zeta$ is 
		$\mu$-separated from $\alpha$ by $\{\beta,\gamma,\delta,\varepsilon\}$.}
		\label{fig:mvctExmp}
	\end{figure}
	
\subsection{Complexity of deciding Markov equivalence}
\label{ssec:complexME}

We have given two characterizations of Markov equivalence of cDGs and argued 
that they both use exponentially many conditions in the worst case. In this 
section, we prove that this, most likely, cannot be circumvented.

coNP is the class of decision problems for which a no-instance can be verified 
using a polynomial-length counterexample 
in polynomial time and a problem is in coNP if and only if its complement is in 
NP. If a 
problem is as hard as any problem in coNP, then we say 
that the problem is coNP-{\it hard}. If a problem is coNP-hard and also in 
coNP, we say that it is coNP-{\it complete} \citep{Garey1979, Sipser2013}. 
Various inference 
problems in graphical models are known to be computationally hard 
\citep{meek2001, chickering2004, chandrasekaran2008, koller2009}. On 
the other hand, there exists polynomial-time algorithms for deciding Markov 
equivalence in several classes of graphs, e.g., maximal ancestral graphs 
\citep{ali2009} and DGs under $d$-separation \citep{richardson1997}. This is 
different in cDGs under $\mu$-separation.

\begin{thm}
	Deciding Markov equivalence of cDGs is coNP-complete.
	\label{thm:meCoNP}
\end{thm}

The complexity result implies that, unless P = coNP (which is commonly believed 
to not be the case), one cannot find a 
characterization of Markov equivalence of cDGs which can be 
verified in polynomial time 
in the size of the graph as this would allow us to decide Markov equivalence 
of two cDGs.

\begin{proof}
	We first show that deciding Markov equivalence is in coNP. This is clear as 
	given two graphs that are not Markov equivalent and a certificate 
	indicating sets $A$,$B$, and $C$ such that we have separation in one but 
	not in the 
	other, we can use Proposition \ref{prop:augCrit} (Appendix \ref{app:aug}) 
	to verify this no-instance 
	in polynomial time.
	
	In order to show that deciding Markov equivalence is coNP-hard, we use a 
	reduction similar to one by \cite{bohler2012} who 
	study 
	complexity of deciding equivalence of Boolean circuits, see in particular 
	the proof of their Lemma 4.3. Consider Boolean 
	variables 
	$x_1,\ldots,x_n$. We say that $x_l$ and $\neg x_l$ are {\it literals}. A 
	Boolean formula is in {\it disjunctive normal form} (DNF) if it is a 
	disjunction of conjuctions of literals. It is a 3DNF, if each conjunction 
	has at most three literals. The 3DNF tautology is the problem of deciding 
	if 
	a 
	3DNF is satisfied for all inputs and this problem is known to be 
	coNP-hard. We reduce 3DNF tautology to the problem 
	of deciding Markov equivalence. Let $H$ be a 3DNF formula on 
	variables $x_1,\ldots,x_n$ consisting of literals
	
	$$
	H = (z_1^1 \wedge z_2^1 \wedge z_3^1) \vee \ldots \vee (z_1^N \wedge 
	z_2^N 
	\wedge z_3^N)
	$$
	
	\noindent such that $z_i^j$ equals $x_l$ or $\neg x_l$ for some $l = 
	1,\ldots,n$. In the former case, we say that $z_i^j$ is a {\it positive} 
	literal, and in the latter that $z_i^j$ is a {\it negative} literal. We say 
	that a conjunction, e.g., $z_1^j \wedge z_2^j \wedge	 z_3^j $, is a {\it 
	term}. 
	In the following, we 
	will define graphs in which the nodes correspond to literals, variables, 
	and negated variables in this problem. We will use Greek alphabet letters 
	for the nodes. Now 
	define 
	
	$$
	V^- = \{\zeta_i^j\} \cup \{\chi_l, \upsilon_l\},
	$$
	
	\noindent such that $\zeta_i^j$ corresponds to $z_i^j$, $\chi_l$ to $x_l$, 
	and $\upsilon_l$ to the negation of $x_l$. We also define
	
	$$
	V = \{\alpha,\beta\} \cup V^- \cup \{\gamma_\delta: \delta \in V^-\}.
	$$
	
	\noindent We construct a cDG on nodes $V$ with the following edge set. We 
	use $\rho_1 \leftrightarrows \rho_2$ to denote that 
	$\rho_1\rightarrow\rho_2$ and $\rho_1\leftarrow\rho_2$. We use 
	$\rho_1,\ldots,\rho_k \unEdge \rho_{k+1},\ldots, \rho_{k+m}$ to denote that 
	there is a blunt edge between any pair $\delta_1 \in 
	\{\rho_1,\ldots,\rho_k\}$ and $\delta_2 \in \{\rho_{k+1},\ldots,
	\rho_{k+m}\}$. Every node 
	has a directed loop. Furthermore, for $\delta\in V^-$,
	
	$$
	\alpha \rightarrow \gamma_\delta \leftrightarrows \delta.
	$$
	
	\noindent For every term (analogously if the term has fewer than three 
	literals),
	
	$$
	\alpha \rightarrow \zeta_1^j \unEdge \zeta_2^j \unEdge \zeta_3^j \unEdge 
	\chi_1
	$$
	
	\noindent and also $\zeta_3^j \unEdge \upsilon_1$. Furthermore, 
	$\chi_l, \upsilon_l 
	\unEdge 
	\chi_{l+1}, \upsilon_{l+1}$, $l = 1,\ldots,n-1$, and $\chi_n,\upsilon_n 
	\unEdge \beta$. We also include $\chi_1 \unEdge \upsilon_1$. Finally, 
	$\chi_l 
	\leftrightarrows \zeta_i^j$ if and only if $z_i^j$ is a
	positive 
	literal of the variable $x_l$ and $\upsilon_l 
	\leftrightarrows 
	\zeta_i^j$ if and only if $z_i^j$ is a negative literal of the variable 
	$x_l$. 
	We let $\D$ denote the cDG on nodes $V$ and with edges as described above. 
	We also define $\D^+$ by adding edges $\alpha\unEdge \chi_1, \upsilon_1$ to 
	$\D$.
	
	We now argue that $H$ is a tautology (that is, true for all inputs) if and 
	only if $\D$ and $\D^+$ are Markov equivalent. Assume that $H$ is a 
	tautology. To argue that $\D$ and $\D^+$ are Markov equivalent it suffices 
	to show that every 
	collider path of $\D^+$ is covered in $\D$ (Theorem \ref{thm:meChar}). 
	Every collider 
	path in $\D^+$ which is not in $\D$ either contains the subpath $\chi_1 
	\unEdge \alpha \unEdge \upsilon_1$ or is of the below form. If it contains 
	$\chi_1 \unEdge \alpha \unEdge \upsilon_1$, then we can substitute this for 
	$\chi_1 \unEdge \upsilon_1$ and obtain a covering path in $\D$. Assume 
	instead a collider path in $\D^+$ of the following form,
	
	$$
	\alpha \unEdge \varepsilon_1 \unEdge \ldots \sim \varepsilon_{k+1}.
	$$
	
	\noindent If $\varepsilon_{k+1}\neq \beta$, then this is covered in $\D$ by 
	$\alpha\rightarrow 
	\gamma_{\varepsilon_{k+1}} \leftrightarrows
	\varepsilon_{k+1}$, or 
	by $\alpha \rightarrow \varepsilon_{k+1}$. 
	Assume instead that $\varepsilon_{k+1} = \beta$. In this case, for all 
	$i=1\ldots,n$ either $\chi_i\in \{\varepsilon_1\ldots,\varepsilon_k\} $ or 
	$\upsilon_i\in \{\varepsilon_1\ldots,\varepsilon_k\}$. Consider now the 
	following 
	assignment of truth values to the variables: $x_l = 1$ if and only if 
	$\chi_l\in \{\varepsilon_1\ldots,\varepsilon_k\} $. By assumption, $H$ is a 
	tautology, so there is a term which equals $1$ for this assignment, say the 
	$j$'th (without loss of generality assuming the the $j$'th term contains 
	three literals),
	
	$$
	z_1^j \wedge z_2^j \wedge z_3^j.
	$$
	
	\noindent If $z_i^j$ is a positive literal, then it must correspond to a 
	$x_l$ such that $\chi_l \in \{\varepsilon_1\ldots,\varepsilon_k\}$, and 
	then in $\D$, 
	$\zeta_i^j$ is a 
	parent of  $\chi_l\in 
	\{\varepsilon_1\ldots,\varepsilon_k\}$. If it is a negative literal, then 
	it must correspond to $x_l$ such that $\chi_l\notin 
	\{\varepsilon_1\ldots,\varepsilon_k\}$. 
	Then $\upsilon_l \in \{\varepsilon_1\ldots,\varepsilon_k\}$, and therefore 
	$\zeta_i^j$ 
	is a parent of $\{\varepsilon_1\ldots,\varepsilon_k\}$. This means that the 
	walk
	
	$$
	\alpha \rightarrow \zeta_1^j  \unEdge \zeta_2^j  \unEdge \zeta_3^j \unEdge 
	\phi_1 
	\unEdge \ldots \phi_n \unEdge \beta,
	$$ 
	
	\noindent where $\phi_l = \chi_l$ if $\chi_l \in 
	\{\varepsilon_1\ldots,\varepsilon_k\}$ and $\phi_l = \upsilon_l\in 
	\{\varepsilon_1\ldots,\varepsilon_k\}$ else, is a covering path in $\D$. 
	This 
	implies that $\D$ and $\D^+$ are Markov equivalent.
	
	On the other hand, assume that $H$ is not a tautology. In this case, there 
	exists some assignment of truth values such that every term of $H$ is $0$, 
	and let 
	$I$ denote this assignment. We now define the following subset of nodes,
	
	$$
	C = \{ \chi_l: x_l = 1 \text{ in } I\} \cup \{ \upsilon_l: 
	x_l = 0 \text{ in } 
	I\}. 
	$$
	
	\noindent We see that for all $l=1,\ldots,n$, either $\chi_l \in C$ or 
	$\upsilon_l 
	\in C$, and this means that $\beta$ is not $\mu$-separated from $\alpha$ by 
	$C$ in $\D^+$. If we consider a term (again, without loss of generality 
	assuming that the term has three literals),
	
	$$
	z_1^j \wedge z_2^j \wedge z_3^j.
	$$
	
	\noindent we know that (under assignment $I$) one of them must equal $0$, 
	say $z_i^j$. If it is a 
	positive literal, then the corresponding variable equals $0$ in the 
	assignment and $\zeta_i^j$ is not an ancestor of $C$. If it is a negative 
	literal, then the 
	corresponding variable $x_l$ equals $1$ in the assignment, and therefore 
	$\upsilon_l$ is not in $C$, and $\zeta_i^j$ is not an ancestor of $C$. 
	In 
	either case, we see that every path
	
	$$
	\alpha \rightarrow \zeta_1^j \unEdge \zeta_2^j \unEdge \zeta_3^j \unEdge 
	\phi_1
	$$
	
	\noindent such that $\phi_1 \in \{\chi_1,\upsilon_1\}$ contains a 
	nonendpoint node which is not an ancestor of $C$. This 
	implies that the collider path in $\D^+$ between $\alpha$ and $\beta$ which 
	traverses exactly the nodes in $C$ is not covered in $\D$ and therefore 
	$\D$ and $\D^+$ are not Markov equivalent 
	(Theorem \ref{thm:meChar}).
	
	The reduction from 3DNF tautology to the Markov equivalence problem is 
	clearly done 
	in polynomial time and is a many-one reduction.
\end{proof}

\section{Conclusion}

We have studied graphs that represent independence structures in stochastic 
processes that are driven by correlated error processes. We have characterized 
their equivalence classes in two ways and proven that deciding equivalence is 
coNP-complete. The characterizations of Markov equivalence do, however, suggest 
subclasses of cDGs in which deciding Markov equivalence is feasible, e.g., in 
cDGs with blunt components of bounded size, or in cDGs such that the length of 
the
shortest blunt path between two nodes is bounded. 

We have also shown a global Markov property in the case of 
Ornstein-Uhlenbeck processes driven by correlated Brownian motions. It is an 
open question if and how this can be extended to other or larger classes 
of continuous-time stochastic processes.

	
	

\begin{appendix}

\section{Augmentation criterion for $\mu$-separation in cDGs}
\label{app:aug}

In this appendix, we argue that we can decide $\mu$-separation in a cDG by 
considering an {\it augmented graph}, a generalization of a so-called {\it 
	moral graph} \citep{cowell1999}. \citet{richardson2002, richardson2003} use 
	a 
similar approach to decide $m$-separation in {\it ancestral graphs} and {\it 
	acyclic directed mixed graphs} ($m$-separation is defined in Definition 
	\ref{def:mSep} in Appendix \ref{app:proofMarkov}).  \citet{didelez2000} 
	uses 
	a moral graph to 
decide 
$\delta$-separation in DGs. 

An {\it undirected graph} is a graph such that every edge is {\it undirected}, 
$\alpha - \beta$. The augmented graph of a cDG is 
the undirected graph where all collider connected pairs of nodes are adjacent 
(omitting loops). Given an undirected graph and three disjoint subsets of nodes 
$A$, 
$B$, and $C$, we say that $A$ and $B$ are {\it separated} by $C$ if every path 
between $\alpha\in A$ and $\beta\in B$ intersects $C$.

\begin{prop}[Augmentation criterion for $\mu$-separation]
	Let $\D=(V,E)$ be a cDG. Let $A,B,C 
	\subseteq V$, and assume that $B = 
	\{\beta_1,\ldots, \beta_j\}$. Let $B^p = \{\beta_1^p, \ldots, \beta_j^p\}$ 
	and define the graph $\D(B)$ with node set 
	$V\disjU B^p$ such that $\mathcal{D}_V = \D$ and
	
	$$
	\alpha \rightarrow_{\D(B)} \beta_i^p \text{ if } \alpha \rightarrow_{\D} 
	\beta_i \text{ and } \alpha\in V, \beta_i\in B.
	$$
	
	\noindent Then $\musepG{A}{B}{C}{\D}$ if and only if $A\setminus C$ and 
	$B^p$ 
	are separated 
	by $C$ in 
	the augmented graph of $\D(B)_{\an(A\cup B^p\cup C)}$. 
	\label{prop:augCrit}
\end{prop}

\begin{proof}
	The proofs of Propositions D.2 and D.4 by \citet{Mogensen2020a} give the 
	result. First one shows that $\musepG{A}{B}{C}{\D}$ holds if and only 
	if $\msepG{A\setminus C}{B^p}{C}{\D(B)}$ holds ($m$-separation is found in 
	Definition \ref{def:mSep} in Appendix \ref{app:proofMarkov}). The second 
	statement is then 
	shown to be 
	equivalent to separation in the relevant augmented graph using Theorem 1 
	in 
	\cite{richardson2003}. \citet{richardson2003} 
	studies acyclic graphs, however, the proof also applies to cyclic graphs as 
	noted in the paper. 
\end{proof}

	
	\section{Proof of Theorem \ref{thm:globalMarkov}}

\label{app:proofMarkov}

We assume $X$ is a regular Ornstein-Uhlenbeck process (see Example 
\ref{exmp:ou}) with drift

$$
\lambda(x) = M(x-\mu)
$$

\noindent and diffusion matrix $\sigma$ and let $\Sigma=\sigma\sigma^T$. We let 
$a = 
-M\mu$. Let $V = U \disjU W$. We will use the following notation similar to 
 that of \cite{liptser1977},
 
 \begin{align}
 s \circ s & = \sigma_{UU}\sigma_{UU}^T + \sigma_{UW}\sigma_{UW}^T \\
 s \circ S & = \sigma_{UU}\sigma_{WU}^T + \sigma_{UW}\sigma_{WW}^T \\
 S \circ S & = \sigma_{WU}\sigma_{WU}^T + \sigma_{WW}\sigma_{WW}^T
 \end{align}
 
 \noindent Note that the above matrices are simply the block components of 
 $\Sigma=\sigma\sigma^T$,
 
 \begin{align}
 \Sigma  = \begin{bmatrix}
 \sigma_{UU} & \sigma_{UW} \\[.1cm]
 \sigma_{WU} & \sigma_{WW}
 \end{bmatrix}
 \begin{bmatrix}
 \sigma_{UU}^T & \sigma_{WU}^T \\[.1cm]
 \sigma_{UW}^T & \sigma_{WW}^T
 \end{bmatrix} =
 \begin{bmatrix}
 s \circ s & s \circ S  \\[.1cm]
 (s \circ S)^T & S \circ S
 \end{bmatrix}.
 \end{align}
 
 We let $m_t$ denote $$E\left(X_t^U \mid \mathcal{F}_t^W\right).$$ The 
 following 
 integral 
 equation holds \cite[Theorem 
 10.3]{liptser1977},
 
 \begin{align}
 m_t &= m_0 + \int_0^t a_U + M_{UU}m_s + M_{UW}X_s^W \md s  \\ &+ \int_0^t 
 (s 
 \circ S + 
 \gamma_s 
 M_{WU}^T) (S\circ S)^{-1}(\md X_s^W - (a_W + M_{WU}m_s + M_{WW}X_s^W)\md 
 s)
 \end{align}

 \noindent where $m_0 = \E{X_0^U}{\mathcal{F}_0^W}$ and $\gamma_t$ is the 
 solution of a differential equation given below. We can write this as
 
 \begin{align*}
 m_t = m_0 + \int_0^t a_U + (M_{UU} + (s \circ S + \gamma_s 
 M_{WU}^T) (S\circ S)^{-1}M_{WU})m_s + M_{UW}X_s^W \md s \\ + \int_0^t (s \circ 
 S + 
 \gamma_s 
 M_{WU}^T) (S\circ S)^{-1}(\md X_s^W - (a_W + M_{WW}X_s^W) \md s).
 \end{align*}

 The process $\gamma(t)$ is given by the following equation \cite[Theorem 
 10.3]{liptser1977}. 
 
 \begin{align}
 \dot{\gamma}(t) &= M_{UU}\gamma(t) + \gamma(t)M_{UU}^T + 
 s \circ s \\
 &-\left(s \circ S + 
 \gamma(t)M_{WU}^T \right) [S\circ S]^{-1}\left(s \circ S + \gamma(t) 
 M_{WU}^T \right)^T \\
 & = (M_{UU} - (s\circ S)[S\circ S]^{-1}M_{WU})\gamma(t) + 
 \gamma(t)(M_{UU}^T - M_{WU}^T[S\circ S]^{-1}(s\circ S)^T ) 
 \\
 &+ 
 s \circ s - (s\circ S)[S\circ S]^{-1}(s\circ S)^T - 
 \gamma(t)M_{WU}^T[S\circ S]^{-1}M_{WU}\gamma(t)
 \end{align}

 \noindent with initial condition $\gamma_0 = \mathrm{E}[(X_0^U - m_0)(X_0^U - 
 m_0)^T]$. This is known as a 
 {\it 
 differential Riccati equation}. The solution 
 of these 
 equations is unique when we restrict our attention to solutions such that 
 $\gamma_t$ is symmetric and nonnegative definite \cite[Theorem 
 10.3]{liptser1977}. Essentially, we will show the global Markov property by 
 arguing about the measurability of $m_t$, using the sparsity of the 
 matrices that go into the integral equation. We will achieve this by first 
 describing the sparsity in the solution of an associated {\it algebraic} 
 Riccati equation and this will allow us to describe the sparsity in the 
 solution of the differential Riccati equation.
 
For ease of notation, we define matrices

\begin{align}
D &= M_{UU}^T - M_{WU}^T[S\circ S]^{-1}(s\circ S)^T,  \\
E &= M_{WU}^T[S\circ S]^{-1}M_{WU}, \\
F &= s \circ s - (s\circ S)[S\circ S]^{-1}(s\circ S)^T,
\end{align}

\noindent and this allows us to write the equation as

$$
\dot{\gamma}(t) = \gamma(t)D + D^T\gamma(t) - \gamma(t)E\gamma(t) + F.
$$

Note that $F$ is the Schur complement of $S\circ S$ in $\Sigma$. The 
matrix $\Sigma$ is positive definite by assumption, and therefore so are 
$F$ \citep[p. 472]{horn1985} and $S\circ S$.

\subsection{Sparsity of the solution of the algebraic Riccati equation}

In order to solve the differential Riccati equation, we will first solve an 
algebraic Riccati equation (Equation (\ref{eq:algRic})) --- or rather argue 
that 
its solution has a certain 
sparsity structure.

\begin{align}
0 = \Gamma D + D^T \Gamma - \Gamma E \Gamma + F
\label{eq:algRic}
\end{align}

The concept of $\mu$-separation 
is similar to that of $m$-separation \cite{spirtes1998, koster1999, 
	richardson2003}
which has been used 
in acyclic graphs.

\begin{defn}[$m$-separation]
	In a graph and for disjoint node sets $A$, $B$, and $C$, we 
	say that $A$ and $B$ are \emph{$m$-separated} given $C$ (and write 
	$\msep{A}{B}{C}$) if there is no path 
	between any $\alpha\in A$ and any $\beta\in B$ such that every collider is 
	in 
	$\an(C)$ and no noncollider is in $C$.
	\label{def:mSep}
\end{defn}

$m$-separation is, in contrast to 
$\mu$-separation, a symmetric notion of separation in the sense that if $B$ is 
$m$-separated from $A$ given $C$, then $A$ is also $m$-separated from $B$ given 
$C$. We will use $m$-separation as a technical tool in our study of cDGs as 
some statements are more easily expressed using this symmetric notion.

In the following proposition and its proof, we write $A \rightharpoondown B 
\mid C$ if 
there exists 
$\alpha\in A$ and $\beta\in B$ such that there is walk between $\alpha$ and 
$\beta$ with every collider in $\an(C)$ and no noncollider in $C$ and 
furthermore there is a neck on the final edge at 
$\beta$. 

\begin{prop}
	Consider a regular Ornstein-Uhlenbeck process. Assume $V = U \disjU W$ and 
	$W = A \disjU C$ and define
	
		\begin{align*}
		V_1 & = \{u\in U : 
		\msep{u}{A}{C}\}, \\
		V_2 & = \{u \in U: \msep{u}{V_1}{A\cup C},\ \notmsep{u}{A}{C} \}, \\
		V_3 & = \{u\in U: \notmsep{u}{V_1}{A\cup C},\ \notmsep{u}{A}{C} 
		\}, \\
		V_4 & = \{w \in W: V_1 \rightharpoondown w \mid W \}, \\
		V_5 & = \{w \in W: V_2 \rightharpoondown w \mid W \}, \\
		V_6 & = W \setminus (V_4 \cup V_5).
		\end{align*}
		
	\noindent If $B$ is $\mu$-separated from $A$ given $C$ in the canonical 
	local 
	independence 
	graph, $\D$, then
	$U = \bar{V_1}\disjU \bar{V_2} \disjU \bar{V_3}, W = V_4 \disjU 
	V_5\disjU V_6$, $\pa_\D(B)\setminus (A\cup C) \subseteq 
	V_1$, and furthermore after a reordering of the rows and 
	columns such that the order is consistent with $V_1, \ldots, V_6$, we have 
	the following sparsity of the matrices $M$ and $\Sigma$,
	
	\begin{align*}
	M & = \begin{bmatrix}
	M_{11} & 0 & 0 & M_{14} & M_{15}  & M_{16} \\
	0 & M_{22} & 0 & M_{24} & M_{25}  & M_{26} \\
	M_{31} & M_{32} & M_{33} & M_{34} & M_{35}  & M_{36} \\
	M_{41} & 0 & 0 & M_{44} & M_{45}  & M_{46} \\ 
	0 & M_{52} & 0 & M_{54} & M_{55}  & M_{56} \\ 
	0 & 0 & 0 & M_{64} & M_{65}  & M_{66} \\ 
	\end{bmatrix},  \ \ \ \\
	\Sigma=\sigma\sigma^T & = \begin{bmatrix}
	\Sigma_{11} & 0 & \Sigma_{13} & \Sigma_{14} & 0  & 
	0 \\
	0 & \Sigma_{22} &  \Sigma_{23} & 0 & \Sigma_{25}  & 
	0 \\
	\Sigma_{31} & \Sigma_{32} & \Sigma_{33} & \Sigma_{34} & \Sigma_{35}  & 
	\Sigma_{36} \\
	\Sigma_{41} & 0 & \Sigma_{43} & \Sigma_{44} & 0  & 0 \\ 
	0 & \Sigma_{52} & \Sigma_{53} & 0 & \Sigma_{55}  & 0 \\ 
	0 & 0 & \Sigma_{63} & 0 & 0  & 
	\Sigma_{66} \\ 
	\end{bmatrix}.
	\end{align*}
	
	\label{prop:sparsAlgRic}
\end{prop}

\noindent For both matrices, the subscript $ij$ corresponds to rows $V_i$ and 
columns $V_j$.

\begin{proof}

	We have 
	that $U = V_1 
	\disjU V_2 \disjU V_3$. If $w \in V_4 \cap V_5 \neq \emptyset$, then there 
	is an $m$-connecting walk between $V_1$ and $V_2$ given $A\cup C$ which 
	would be a 
	contradiction, and thus, $W = V_4\disjU V_5\disjU V_6$. Note that 
	$\Sigma$ is symmetric so we only need to argue that the lower 
	triangular part has the postulated sparsity pattern. Whenever we mention an 
	$m$-connecting 
	walk in this proof without specifying a conditioning set we tacitly mean 
	`given $W$'.
	
	Any edge $V_1 \sim V_2$ would create an $m$-connecting walk and therefore 
	$M_{21} = 0, M_{12} = 0, \Sigma_{21} = 0$. An edge $V_1 
	\rightarrow w \in V_5$ would also create an $m$-connecting walk between 
	$V_1$ and $V_2$ as $V_5 \subseteq W$, and therefore $M_{51} = 0$. 
	Similarly, we see that $M_{42} = 0$, $\Sigma_{51} = 0$, and $\Sigma_{42} = 
	0$. 
	If $V_1 \rightarrow w \in V_6$, then $w$ would have to be in $V_4$, and 
	thus, 
	$M_{61} = 0$. Similarly, $M_{62}=0$, $\Sigma_{61} = 0$, $\Sigma_{62} = 0$. 
	Let $u \in V_3$. Then there exists an $m$-connecting walk between $u$ and 
	$A$ given $C$, and composing this walk with an edge $u \rightarrow V_1$ 
	would give 
	an $m$-connecting walk between $A$ and $V_1$ given $C$ as $u\notin C$. This 
	is a 
	contradiction and $M_{13} = 0$. Similarly, $M_{23} = 0$, using the 
	$m$-connecting walk between $u$ and $V_1$. Consider again $u 
	\in V_3$. There exists $m$-connecting walks between $u$ and $V_1$ (given 
	$A\cup C$) and $u$ 
	and $A$ (given $C$). None of them can have a tail at $u$ as otherwise we 
	could find an $m$-connecting walk between $A$ and $V_1$ given $C$. 
	Therefore, $u$ is a collider on their 
	composition, and from this it follows that $M_{43} = 0, M_{53} = 0, M_{63} 
	= 0$. If $V_4 \unEdge V_5$, it would follow that there is an $m$-connecting 
	walk 
	between $A$ and $V_1$, a contradiction. It follows that $\Sigma_{54} = 
	0$. If $V_4 \unEdge w\in W$, then $w \in V_4$, and it follows that 
	$\Sigma_{64} 
	= 0$. 
	Similarly, $\Sigma_{65} = 0$.
\end{proof}

The matrices $D,E$, and $F$ all have their rows and columns indexed by $U = 
V_1\disjU V_2 \disjU V_3$. The 
above proposition and the definitions of the matrices $D,E$, and $F$ give 
the following.

\begin{cor}
	Under the conditions of Theorem \ref{thm:globalMarkov}, the matrix $D$ has 
	the 
	sparsity structure
	
	$$
	\begin{bmatrix}
	* & 0 & * \\
	0 & * & * \\
	0 & 0 & *
	\end{bmatrix},
	$$
	
	\noindent i.e., $D_{V_2V_1} = 0, D_{V_3V_1} = 0, D_{V_1V_2} = 0$, and 
	$D_{V_3V_2} = 0$. The matrix $F$ is such that $F_{V_1V_2} = 0$ and 
	$F_{V_2V_1}=0$. 
	The matrix 
	$E$ is block diagonal and $E_{V_3V_3} = 0$.
	\label{cor:sparsDEF}
\end{cor}

\begin{lem}
	If $N$ is an invertible matrix with the sparsity of $D$, then so is 
	$N^{-1}$.
	\label{lem:sparsDinv}
\end{lem}

\begin{proof}
	The matrices on the (block) diagonal of $N$ must also be invertible, and 
	the result follows 
	from the Schur complement representation of $N^{-1}$, 
	using the first two blocks as one component, and the third as the second 
	component.
\end{proof}

\begin{lem}
	Consider the Lyapunov equation for square matrices $L,Z,$ and $Q$ such that 
	$Q$ is symmetric,
	
	$$
	LZ + ZL^T + Q = 0,
	$$
	
	\noindent and let $Z_0$ denote its solution. If $L$ is stable and has the 
	sparsity 
	pattern of $D^T$ and $Q$ is such that 
	$Q_{V_1V_2} = 0$, $Q_{V_2V_1} = 0$, 
	then $(Z_0)_{V_1V_2} = 0$ and $(Z_0)_{V_2V_1} = 0$.
	\label{lem:lyapSol}
\end{lem}

\begin{proof}
	The result follows from the explicit solution of a Lyapunov equation when 
	$L$ is stable \citep{lancaster1995},
	
	$$
	Z_0 = \int_0^\infty e^{Ls } Q e^{L^T s} \md s.
	$$
\end{proof}

\begin{defn}[Stabilizable pair of matrices]
	Let $G$ and $H$ be matrices, $n\times n$ and $n\times m$, respectively. We 
	say that the pair $(G,H)$ is \emph{stabilizable} if there exists an 
	$m\times 
	n$ matrix, $K$, such that $G + HK$ is stable.
\end{defn}

In the literature, stabilizability is used in both the context of 
continuous-time and discrete-time systems. The above definition is that of a 
continuous-time system \citep[p. 90]{lancaster1995}. The following 
is 
proven in \cite{jacob2012}.

\begin{lem}
	The pair $(A,B)$ is stabilizable if and only if for every eigenvector of 
	the matrix $A^T$ with eigenvalue $\lambda$ such that $Re(\lambda) \geq 0$ 
	it holds that $v^T B \neq 0$.
	\label{lem:stabABcrit} 
\end{lem}

\begin{lem}
	The pair $(D,E)$ is stabilizable.
	\label{lem:DEstab}
\end{lem}

\begin{proof}
	We will prove this using Lemma \ref{lem:stabABcrit}. To obtain a 
	contradiction, assume that there exists an eigenvector $v$ of $D^T$ with 
	corresponding eigenvalue $\lambda$ such that $Re(\lambda) \geq 0$, and 
	assume furthermore that $v^T E = 0$. The matrix $(S \circ S)^{-1}$ is 
	positive definite (since $\Sigma$ is positive definite), and $v^T 
	M_{WU}^T (S \circ S)^{-1} M_{WU} v = 0$. It 
	follows 
	that $M_{WU} v = 0$. Let $o$ be the column vector of zeros of length 
	$l$. Note that 
	$\lambda v = 
	D^T v =M_{UU} v$. Then,
	
	$$
	M
	\begin{pmatrix}
	v \\ o
	\end{pmatrix}
	=
	\begin{pmatrix}
	M_{UU} & M_{UW} \\ M_{WU} & M_{WW}
	\end{pmatrix}
	\begin{pmatrix}
	v \\ o
	\end{pmatrix}
	= \lambda 
	\begin{pmatrix}
	v \\ o
	\end{pmatrix}
	$$
	
	\noindent It follows that $\lambda$ is an eigenvalue of $M$ which is a 
	contradiction 
	as $M$ is stable by assumption.
\end{proof}

\begin{cor}
	There exists a symmetric $k\times k$ matrix $X_0$ such that $(X_0)_{V_1V_2} 
	= 0$, $(X_0)_{V_2V_1}=0$ and such that $D - EX_0$ is 
	stable.
	\label{cor:X0}
\end{cor}

We let $k$ denote the cardinality of $U$.

\begin{proof}
	From the above lemma it follows that there exists a $k\times k$ matrix 
	$\bar{X}$ such that $D + E\bar{X}$ is stable. From the sparsity of $D$ and 
	$E$ it follows that for any $k\times k$ matrix, $X$, $D + EX$ is stable if 
	and only if $D_{\{V1,V2\}\{V1,V2\}} + 
	E_{\{V1,V2\}\{V1,V2\}}X_{\{V1,V2\}\{V1,V2\}}$ and $D_{V_3V_3}$ are stable. 
	The matrices $D_{\{V1,V2\}\{V1,V2\}}$ and 
	$E_{\{V1,V2\}\{V1,V2\}}$ are both block diagonal and thus both pairs of 
	blocks are 
	stabilizable using the existence of $\bar{X}$ and Lemma 
	\ref{lem:stabABcrit}. It follows that 
	$X_0$ 
	can be chosen as block diagonal. We need to argue that $X_0$ can be chosen 
	to be symmetric. The blocks in the diagonal of $E$ are positive 
	semidefinite and
	stabilizable (when paired with their corresponding $D$ blocks). Therefore 
	$X_0$ can be chosen to also be positive semidefinite (and therefore 
	symmetric) and such that $D - EX_0$ is stable 
	\cite[Lemma 4.5.4]{lancaster1995}, see also \cite{guo1998}. 
\end{proof}

Matrices $E$ and $F$ are both positive semidefinite and there exist unique 
positive semidefinite matrices $\bar{E}$ and $\bar{F}$ such that $E = 
\bar{E}\bar{E}$ and such that $F = \bar{F} \bar{F}$ \citep[Theorem 
7.2.6]{horn1985}.

\begin{cor}
	The pair $(D,\bar{E})$ is stabilizable.
	\label{cor:DbarEstab}
\end{cor}

\begin{proof}
	This follows from the fact that $(D,E)$ is stabilizable (Lemma 
	\ref{lem:DEstab}).
\end{proof}

\begin{defn}[Detectable pair of matrices]
	Let $G$ and $H$ be matrices, $m\times n$ and $n\times n$ respectively. We 
	say that the pair $(G,H)$ is \emph{detectable} if there exists an $n\times 
	m$ matrix, $X$, such that $XG + H$ is stable.
\end{defn}

\begin{prop}
	The pair $(\bar{F},D)$ is detectable. The pair $(F,D)$ is also detectable.
	\label{prop:FDdetect}
\end{prop}

\begin{proof}
	Observe that $\bar{F}$ is invertible. This means that we can choose $X = 
	(-I - D)\bar{F}^{-1}$. With this choice of $X$, the matrix $X\bar{F} + D$ 
	is stable. A similar argument works for the pair $(F, D)$.
\end{proof}

We argue now that there is a unique nonnegative definite solution of the 
algebraic 
Riccati equation (\ref{eq:algRic}) by showing that the conditions of Theorem 2 
in \cite{kucera1972} are fulfilled. The pair $(D,\bar{E})$ is 
stabilizable (Corollary \ref{cor:DbarEstab}) and the pair $(\bar{F}, D)$ is 
detectable (Proposition \ref{prop:FDdetect}) and we just need to show that 

	$$
	\tilde{M}
	=
	\begin{pmatrix}
	D & \ & -E \\ -F & \ & -D^T
	\end{pmatrix}
	$$

\noindent is such that $Re(\lambda) \neq 0$ for all eigenvalues, $\lambda$, of 
$\tilde{M}$. Assume to obtain a contradiction that $\lambda$ is a eigenvalue 
of $\tilde{M}$ such that 
$Re(\lambda) = 0$,

$$
\lambda v = \tilde{M} v, \ \ \ \ 
	v
	=
	\begin{pmatrix}
	v_1 \\ v_2
	\end{pmatrix}.
$$

\noindent Similarly to what is done in \cite{molinari1973}, we left-multiply by 
$(
v_2^* \ \  v_1^*)$  where $*$ denotes conjugate transpose to obtain

	$$
	\begin{pmatrix}
	v_2^* & \ & v_1^* 
	\end{pmatrix}
	\begin{pmatrix}
	D & \ & -E \\ -F & \ & -D^T
	\end{pmatrix}
		\begin{pmatrix}
		v_1 \\ v_2
		\end{pmatrix}
	=
	\lambda (v_2^*v_1 + v_1^*v_2).
	$$
	
\noindent By taking real parts on both sides of the above equation, we obtain
$Re(-v_2^* E v_2 - v_1^*Fv_1) = 0$. Matrices $E$ and $F$ are both positive 
semidefinite so $v_2^* E v_2 = 0$ and $v_1^*Fv_1 = 0$. The matrix $F$ is 
positive definite so $v_1 = 0$. Lemma \ref{lem:stabABcrit} gives a 
contradiction to the fact that $(D,E)$ is stabilizable. In conclusion, it 
follows from Theorem 2 in \cite{kucera1972} that there exists a unique 
positive 
semidefinite 
solution of the algebraic Riccati 
equation.

\begin{lem}[Sparsity in solution of algebraic Riccati equation]
	Under the conditions of Theorem \ref{thm:globalMarkov}, it holds that 
	$\bar{\Gamma}_{V_1 V_2} = 0$ when $\bar{\Gamma}$ is the 
	unique, nonnegative definite solution of Equation (\ref{eq:algRic}). 
	\label{lem:sparsAlgRic}
\end{lem}

\begin{proof}
	Theorem 1 of \citet{guo1998} applies as $F$ is positive semidefinite. We 
	know from above that there is a unique positive semidefinite solution and 
	this must necessarily be the same as the maximal symmetric solution of 
	Theorem 1 in \cite{guo1998}.
	
	Using Corollary \ref{cor:X0}, there exists a symmetric $k\times k$ matrix, 
	$X_0$, such that $(X_0)_{V_1V_2} = 0$,  $(X_0)_{V_2V_1} = 0$, and such that 
	$D - EX_0$ is stable. From this matrix, we will define a sequence of 
	matrices that converge to $X_+$.  With this purpose in mind, we define a 
	Newton step 
	as the operation that takes a matrix
	$X_i$ to the solution of (this is an equation in $X$)
	
	$$
	(D - EX_i)^T X + X(D-EX_i) + X_iEX_i + F = 0.
	$$
	
	Assume now that $X_i$ is such that $(X_i)_{V_1V_2} = 0$ and $(X_i)_{V_2V_1} 
	= 0$. Note first that by Corollary \ref{cor:sparsDEF}, $\bar{Q} = X_iEX_i + 
	F$ is also such that 
	$\bar{Q}_{V_1V_2} = 0$ and $\bar{Q}_{V_2V_1} = 0$. The matrix $EX_i$ has 
	the 
	sparsity pattern of $D$ (i.e., $D_{jk} = 0$ implies that $(EX_i)_{jk} = 
	0$), and the matrix $D$ does too. By induction and 
	using Lemma \ref{lem:lyapSol}, it 
	follows that $X_i$ is such that $(X_i)_{V_1V_2} = 0$ and 
	$(X_i)_{V_2V_1} = 0$ for all $i\geq 0$. Note that for all $i$ it holds that 
	$D - EX_i$ is stable and that $X_i$ is symmetric \citep{guo1998}. Theorem 
	1.2 of \citet{guo1998} now 
	gives that $X_+ = \lim X_i$ is the solution of the 
	algebraic Riccati equation, and it follows from the above that 
	$(X_+)_{V_1V_2} = 0$ and $(X_+)_{V_2V_1} = 0$.	
\end{proof}

\subsection{Sparsity in the solution of the differential Riccati equation}

We will use the above results on the algebraic Riccati equation to describe 
zero entries of the solution to the differential Riccation equation. From 
\cite{choi1990}, it follows that if $\Gamma_0$ is positive definite, then

\begin{align}
\Gamma(t) = \bar{\Gamma} + \text{e}^{tK^T}(\Gamma_0 - \bar{\Gamma})\left(I + 
\int_0^t \text{e}^{sK}E\text{e}^{sK^T} \md s (\Gamma_0 - \bar{\Gamma}) 
\right)^{-1} 
\text{e}^{tK}
\label{eq:diffRic}
\end{align}

\noindent where $K = D-E\bar{\Gamma}$ and $\bar{\Gamma}$ is the 
unique 
nonnegative definite solution of the algebraic Riccati equation (Equation 
(\ref{eq:algRic})). 

\begin{proof}[Proof of Equation \ref{eq:diffRic}]
	From \cite{choi1990}, we have that Equation (\ref{eq:diffRic}) holds under 
	whenever $\Gamma_0$ is positive definite as $(D, \bar{E})$ is stabilizable 
	(Corollary 
	\ref{cor:DbarEstab}), and $(\bar{F}, D)$ 
	is detectable (Proposition \ref{prop:FDdetect}).
\end{proof}

\begin{lem}
	Let $\Gamma(t)$ denote the solution of the 
	differential Riccati equation (Equation (\ref{eq:diffRic})) with initial 
	condition $\Gamma_0$. Under the conditions of Theorem 
	\ref{thm:globalMarkov}, it holds that $(\Gamma(t))_{V_1V_2} = 0$ for all $t 
	\geq 0$.
\end{lem}

\begin{proof}
	This follows directly from the expression in Equation (\ref{eq:diffRic}) 
	and 
	the sparsity of the matrices that go into that expression: $\text{e}^{tK}$ 
	has the sparsity of $D$ and $\text{e}^{tK^T}$ has that of $D^T$. From Lemma 
	\ref{lem:sparsAlgRic} we know that 
	$\bar{\Gamma}_{V_1V_2} = 0$. The matrix
	
	$$
	I + 
	\int_0^t \text{e}^{sK}E\text{e}^{sK^T} \md s (\Gamma_0 - \bar{\Gamma})
	$$
	
	\noindent has the sparsity of $D$ and so does its inverse (Lemma
	\ref{lem:sparsDinv}). This result follows immediately by matrix 
	multiplication.
\end{proof}

\begin{proof}[Proof of Theorem \ref{thm:globalMarkov}]
	Let $\beta\in B$ and let $t \in I$. We need to show that 
	$$E\left(\lambda_t^\beta 
	\mid 
	\mathcal{F}_t^{A\cup C}\right)$$ is almost surely equal to an 
	$\mathcal{F}_t^C$-measurable random variable. We can without loss of 
	generality assume 
	that $A$ and $C$ are disjoint. The fact that $B$ is $\mu$-separated from 
	$A$ given $C$ implies that $M_{\beta A} = 0$,
	
	\begin{align*}
	E\left(\lambda_t^\beta \mid 
	\mathcal{F}_t^{A\cup C}\right) & = - M_{\beta V}\mu + \sum_{\gamma\in A\cup 
	C} M_{\beta\gamma} 
	X_t^\gamma + \sum_{\delta\notin A\cup C} M_{\beta\delta}E\left(X_t^\delta 
	\mid 
	\mathcal{F}_t^{A\cup C}\right) \\
	& = - M_{\beta V}\mu + \sum_{\gamma\in C} M_{\beta\gamma} 
	X_t^\gamma + \sum_{\delta \in \pa_\D(\beta) \setminus (A\cup C)} 
	M_{\beta\delta}E\left(X_t^\delta 
	\mid 
	\mathcal{F}_t^{A\cup C}\right)
	\end{align*}
	
	\noindent where $\D$ is the canonical local independence graph. Let $U = 
	V\setminus A\cup C$. Consider now the partition of $V$ 
	given
	in Proposition \ref{prop:sparsAlgRic}. We see 
	that $\pa_\D(\beta) \setminus (A\cup C) \subseteq V_1$. 
	The matrix $M_{UU} + (s\circ S + \gamma_t M_{WU}^T)(S\circ S)^{-1}M_{WU}$ 
	in 
	the integral equation for the conditional expectation process has the 
	sparsity of $D^T$ and it follows that one can solve for $m_t^{V_1}$ 
	independently of $m_t^{U\setminus V_1}$ as the solution of the smaller 
	system 
	is unique and continuous \citep{liptser1977, beesack1985}. We see that 
	processes 
	$X_t^A$ do not enter into these 
	equations. This follows from the sparsity of $s\circ S$, $S\circ S$, and of 
	$\gamma_t M_{WU}^T$, and the fact that
	$M_{V_4A}=0$ and $M_{V_1 A} =0$, noting that $A \cap V_4 = \emptyset$.
\end{proof}

\end{appendix}




\section*{Acknowledgements}
This work was supported by VILLUM FONDEN (research grant 13358).


\bibliographystyle{imsart-number} 
\bibliography{linearSDEs}       

\begin{thebibliography}{63}
\providecommand{\natexlab}[1]{#1}
\providecommand{\url}[1]{\texttt{#1}}
\expandafter\ifx\csname urlstyle\endcsname\relax
  \providecommand{\doi}[1]{doi: #1}\else
  \providecommand{\doi}{doi: \begingroup \urlstyle{rm}\Url}\fi

\bibitem[Aalen(1987)]{aalen1987}
Odd~O. Aalen.
\newblock Dynamic modelling and causality.
\newblock \emph{Scandinavian Actuarial Journal}, pages 177--190, 1987.

\bibitem[Aalen and Gjessing(2004)]{aalen2004}
Odd~O. Aalen and H{\aa}kon~K. Gjessing.
\newblock Survival models based on the {O}rnstein-{U}hlenbeck process.
\newblock \emph{Lifetime Data Analysis}, 10\penalty0 (4):\penalty0 407--423,
  2004.

\bibitem[Aalen et~al.(2012)Aalen, R{\o}ysland, Gran, and
  Ledergerber]{aalen2012}
Odd~O. Aalen, Kjetil R{\o}ysland, Jon~Michael Gran, and Bruno Ledergerber.
\newblock Causality, mediation and time: {A} dynamic viewpoint.
\newblock \emph{Journal of the Royal Statistical Society, Series A},
  175\penalty0 (4):\penalty0 831--861, 2012.

\bibitem[Ali et~al.(2009)Ali, Richardson, and Spirtes]{ali2009}
Ayesha~R. Ali, Thomas~S. Richardson, and Peter Spirtes.
\newblock Markov equivalence for ancestral graphs.
\newblock \emph{The Annals of Statistics}, 37\penalty0 (5B):\penalty0
  2808--2837, 2009.

\bibitem[Bartoszek et~al.(2017)Bartoszek, Gl{\'e}min, Kaj, and
  Lascoux]{bartoszek2017}
Krzysztof Bartoszek, Sylvain Gl{\'e}min, Ingemar Kaj, and Martin Lascoux.
\newblock Using the {O}rnstein-{U}hlenbeck process to model the evolution of
  interacting populations.
\newblock \emph{Journal of Theoretical Biology}, 429:\penalty0 35--45, 2017.

\bibitem[Beesack(1985)]{beesack1985}
Paul Beesack.
\newblock Systems of multidimensional {V}olterra integral equations and
  inequalities.
\newblock \emph{Nonlinear Analysis}, 9\penalty0 (12):\penalty0 1451--1486,
  1985.

\bibitem[B{\"o}hler et~al.(2012)B{\"o}hler, Creignou, Galota, Reith, Schnoor,
  and Vollmer]{bohler2012}
Elmar B{\"o}hler, Nadia Creignou, Matthias Galota, Steffen Reith, Henning
  Schnoor, and Heribert Vollmer.
\newblock Complexity classifications for different equivalence and audit
  problems for {B}oolean circuits.
\newblock \emph{Logical Methods in Computer Science}, 8\penalty0 (3:27), 2012.

\bibitem[Bormetti et~al.(2010)Bormetti, Cazzola, and Delpini]{bormetti2010}
Giacomo Bormetti, Valentina Cazzola, and Danilo Delpini.
\newblock Option pricing under {O}rnstein-{U}hlenbeck stochastic volatility:
  {A} linear model.
\newblock \emph{International Journal of Theoretical and Applied Finance},
  13\penalty0 (7):\penalty0 1047--1063, 2010.

\bibitem[Chandrasekaran et~al.(2008)Chandrasekaran, Srebro, and
  Harsha]{chandrasekaran2008}
Venkat Chandrasekaran, Nathan Srebro, and Prahladh Harsha.
\newblock Complexity of inference in graphical models.
\newblock In \emph{Proceedings of the 24th Conference on Uncertainty in
  Artificial Intelligence (UAI)}, 2008.

\bibitem[Chickering et~al.(2004)Chickering, Heckerman, and
  Meek]{chickering2004}
David~Maxwell Chickering, David Heckerman, and Christopher Meek.
\newblock Large-sample learning of {B}ayesian networks is {NP}-hard.
\newblock \emph{Journal of Machine Learning Research}, 5:\penalty0 1287--1330,
  2004.

\bibitem[Choi(1990)]{choi1990}
Chiu~H. Choi.
\newblock A survey of numerical methods for solving matrix {R}iccati
  differential equations.
\newblock In \emph{IEEE Proceedings on Southeastcon}, 1990.

\bibitem[Commenges and G{\'e}gout-Petit(2009)]{Commenges:2009}
Daniel Commenges and Anne G{\'e}gout-Petit.
\newblock A general dynamical statistical model with causal interpretation.
\newblock \emph{Journal of the Royal Statistical Society. Series B (Statistical
  Methodology)}, 71\penalty0 (3):\penalty0 719--736, 2009.

\bibitem[Cormen et~al.(2009)Cormen, Leiserson, Rivest, and Stein]{introAlgo}
Thomas~H. Cormen, Charles~E. Leiserson, Ronald~L. Rivest, and Clifford Stein.
\newblock \emph{Introduction to Algorithms}.
\newblock Cambridge, MA: MIT Press, third edition, 2009.

\bibitem[Cowell et~al.(1999)Cowell, Dawid, Lauritzen, and
  Spiegelhalter]{cowell1999}
Robert~G. Cowell, Philip Dawid, Steffen~L. Lauritzen, and David~J.
  Spiegelhalter.
\newblock \emph{Probabilistic Networks and Expert Systems}.
\newblock New York: Springer, 1999.

\bibitem[Cox and Wermuth(1993)]{cox1993}
D.~R. Cox and Nanny Wermuth.
\newblock Linear dependencies represented by chain graphs.
\newblock \emph{Statistical Science}, 8\penalty0 (3):\penalty0 204--218, 1993.

\bibitem[Danks and Plis(2013)]{danks2013}
David Danks and Sergey Plis.
\newblock Learning causal structure from undersampled time series.
\newblock In \emph{JMLR: Workshop and Conference Proceedings}, volume~10, pages
  1--10, 2013.

\bibitem[Didelez(2000)]{didelez2000}
Vanessa Didelez.
\newblock \emph{Graphical Models for Event History Analysis based on Local
  Independence}.
\newblock PhD thesis, Universit{\"a}t Dortmund, 2000.

\bibitem[Didelez(2006)]{didelez2006}
Vanessa Didelez.
\newblock Graphical models for composable finite {M}arkov processes.
\newblock \emph{Scandinavian Journal of Statistics}, 34\penalty0 (1):\penalty0
  169--185, 2006.

\bibitem[Didelez(2008)]{didelez2008}
Vanessa Didelez.
\newblock Graphical models for marked point processes based on local
  independence.
\newblock \emph{Journal of the Royal Statistical Society, Series B},
  70\penalty0 (1):\penalty0 245--264, 2008.

\bibitem[Ditlevsen and Lansky(2005)]{ditlevsen2005}
Susanne Ditlevsen and Petr Lansky.
\newblock Estimation of the input parameters in the {O}rnstein-{U}hlenbeck
  neuronal model.
\newblock \emph{Physical Review E}, 71, 2005.

\bibitem[Eichler(2007)]{eichlerGranger2007}
Michael Eichler.
\newblock Granger causality and path diagrams for multivariate time series.
\newblock \emph{Journal of Econometrics}, 137:\penalty0 334--353, 2007.

\bibitem[Eichler(2012{\natexlab{a}})]{eichler2012}
Michael Eichler.
\newblock Graphical modelling of multivariate time series.
\newblock \emph{Probability {T}heory and {R}elated {F}ields}, 153\penalty0
  (1):\penalty0 233--268, 2012{\natexlab{a}}.

\bibitem[Eichler(2012{\natexlab{b}})]{eichlerChapter2012}
Michael Eichler.
\newblock Causal inference in time series analysis.
\newblock In Carlo Berzuini, Philip Dawid, and Luisa Bernardinelli, editors,
  \emph{Causality: {S}tatistical perspectives and applications}, pages
  327--354. New York: John Wiley \& Sons, 2012{\natexlab{b}}.

\bibitem[Eichler(2013)]{eichler2013}
Michael Eichler.
\newblock Causal inference with multiple time series: {P}rinciples and
  problems.
\newblock \emph{Philosophical Transactions of the Royal Society}, 371\penalty0
  (1997):\penalty0 1--17, 2013.

\bibitem[Eichler and Didelez(2007)]{eichler2007}
Michael Eichler and Vanessa Didelez.
\newblock Causal reasoning in graphical time series models.
\newblock In \emph{{P}roceedings of the 23rd {C}onference on {U}ncertainty in
  {A}rtificial {I}ntelligence (UAI)}, pages 109--116, 2007.

\bibitem[Eichler and Didelez(2010)]{eichler2010}
Michael Eichler and Vanessa Didelez.
\newblock On {G}ranger causality and the effect of interventions in time
  series.
\newblock \emph{Lifetime Data Analysis}, 16\penalty0 (1):\penalty0 3--32, 2010.

\bibitem[Garey and Johnson(1979)]{Garey1979}
Michael~R. Garey and David~S. Johnson.
\newblock \emph{Computers and Intractability: {A} Guide to the Theory of
  {NP}-Completeness}.
\newblock 1979.

\bibitem[Granger and Newbold(1986)]{grangerForecast1986}
Clive W.~J. Granger and Paul Newbold.
\newblock \emph{Forecasting economic time series}.
\newblock Academic Press, 2nd edition, 1986.

\bibitem[Guo and Lancaster(1998)]{guo1998}
Chun-Hua Guo and Peter Lancaster.
\newblock Analysis and modification of {N}ewton's method for algebraic
  {R}iccati equations.
\newblock \emph{Mathematics of Computation}, 67\penalty0 (223), 1998.

\bibitem[Heath(2000)]{heath2000}
Richard~A. Heath.
\newblock The {O}rnstein-{U}hlenbeck model for decision time in cognitive
  tasks: {A}n example of control of nonlinear network dynamics.
\newblock \emph{Psychological Research}, 63:\penalty0 183--191, 2000.

\bibitem[Horn and Johnson(1985)]{horn1985}
Roger~A. Horn and Charles~R. Johnson.
\newblock \emph{Matrix Analysis}.
\newblock Cambridge University Press, 1985.

\bibitem[Hyttinen et~al.(2016)Hyttinen, Plis, J\"arvisalo, Eberhardt, and
  Danks]{Hyttinen2016}
Antti Hyttinen, Sergey Plis, Matti J\"arvisalo, Frederick Eberhardt, and David
  Danks.
\newblock Causal discovery from subsampled time series data by constraint
  optimization.
\newblock In \emph{Proceedings of the Eighth International Conference on
  Probabilistic Graphical Models (PGM)}, volume~52, pages 216--227, 2016.

\bibitem[Jacob and Zwart(2012)]{jacob2012}
Birgit Jacob and Hans~J. Zwart.
\newblock \emph{Linear Port-{H}amiltonian Systems on Infinite-dimensional
  Spaces}.
\newblock Birkh{\"a}user, 2012.

\bibitem[Jacobsen(1993)]{Jacobsen:1993}
M.~Jacobsen.
\newblock A brief account of the theory of homogeneous {G}aussian diffusions in
  finite dimensions.
\newblock In H.~Niemi et~al., editors, \emph{Frontiers in Pure and Applied
  Probability}, volume~1, pages 86--94, 1993.

\bibitem[Koller and Friedman(2009)]{koller2009}
Daphne Koller and Nir Friedman.
\newblock \emph{Probabilistic Graphical Models: {P}rinciples and Techniques}.
\newblock Cambridge, MA: MIT Press, 2009.

\bibitem[Koster(1999)]{koster1999}
Jan~T.A. Koster.
\newblock On the validity of the {M}arkov interpretation of path diagrams of
  {G}aussian structural equations systems with correlated errors.
\newblock \emph{Scandinavian Journal of Statistics}, 26:\penalty0 413--431,
  1999.

\bibitem[Ku{\v c}era(1972)]{kucera1972}
Vladim\'{i}r Ku{\v c}era.
\newblock On nonnegative definite solutions to matrix quadratic equations.
\newblock \emph{Automatica}, 8:\penalty0 413--423, 1972.

\bibitem[Lancaster and Rodman(1995)]{lancaster1995}
Peter Lancaster and Leiba Rodman.
\newblock \emph{Algebraic {R}iccati Equations}.
\newblock Oxford: Clarendon Press, 1995.

\bibitem[Lauritzen(1996)]{lauritzen1996}
Steffen Lauritzen.
\newblock \emph{Graphical Models}.
\newblock Oxford: Clarendon Press, 1996.

\bibitem[Lee and Whitmore(2006)]{lee2006}
Mei-Ling~Ting Lee and G.~A. Whitmore.
\newblock Threshold regression for survival analysis: {M}odeling event times by
  a stochastic process reaching a boundary.
\newblock \emph{Statistical Science}, 21\penalty0 (4):\penalty0 501--513, 2006.

\bibitem[Liptser and Shiryayev(1977)]{liptser1977}
R.S. Liptser and A.N. Shiryayev.
\newblock \emph{Statistics of Random Processes I: General Theory}.
\newblock New York: Springer-Verlag, 1977.

\bibitem[Maathuis et~al.(2018)Maathuis, Drton, Lauritzen, and
  Wainwright]{maathuis2018}
Marloes Maathuis, Mathias Drton, Steffen Lauritzen, and Martin Wainwright,
  editors.
\newblock \emph{Handbook of graphical models}.
\newblock CRC Press, 2018.

\bibitem[Meek(2001)]{meek2001}
Christopher Meek.
\newblock Finding a path is harder than finding a tree.
\newblock \emph{Journal of Artificial Intelligence Research}, 15:\penalty0
  383--389, 2001.

\bibitem[Mogensen and Hansen(2020)]{Mogensen2020a}
S{\o}ren~Wengel Mogensen and Niels~Richard Hansen.
\newblock Markov equivalence of marginalized local independence graphs.
\newblock \emph{The Annals of Statistics}, 48\penalty0 (1), 2020.

\bibitem[Mogensen et~al.(2018)Mogensen, Malinsky, and Hansen]{mogensenUAI2018}
S{\o}ren~Wengel Mogensen, Daniel Malinsky, and Niels~Richard Hansen.
\newblock Causal learning for partially observed stochastic dynamical systems.
\newblock In \emph{Proceedings of the 34th conference on Uncertainty in
  Artificial Intelligence (UAI)}, 2018.

\bibitem[Molinari(1973)]{molinari1973}
B.~P. Molinari.
\newblock The stabilizing solution of the algebraic {R}iccati equation.
\newblock \emph{SIAM Journal on Control}, 11\penalty0 (2):\penalty0 262--271,
  1973.

\bibitem[Pavliotis(2014)]{Pavliotis:2014}
G.A. Pavliotis.
\newblock \emph{Stochastic Processes and Applications: Diffusion Processes, the
  Fokker-Planck and Langevin Equations}.
\newblock Texts in Applied Mathematics. New York: Springer, 2014.

\bibitem[Ricciardi and Sacerdote(1979)]{ricciardi1979}
Luigi~M. Ricciardi and Laura Sacerdote.
\newblock The {O}rnstein-{U}hlenbeck process as a model for neuronal activity.
\newblock \emph{Biological Cybernetics}, 35:\penalty0 1--9, 1979.

\bibitem[Richardson(1996{\natexlab{a}})]{richardson1996}
Thomas~S. Richardson.
\newblock A discovery algorithm for directed cyclic graphs.
\newblock In \emph{Proceedings of the 12th Conference on Uncertainty in
  Artificial Intelligence (UAI)}, 1996{\natexlab{a}}.

\bibitem[Richardson(1996{\natexlab{b}})]{richardsonPolyAlgo1996}
Thomas~S. Richardson.
\newblock A polynomial-time algorithm for deciding {M}arkov equivalence of
  directed cyclic graphical models.
\newblock In \emph{Proceedings of the 12th Conference on Uncertainty in
  Artificial Intelligence (UAI)}, 1996{\natexlab{b}}.

\bibitem[Richardson(1997)]{richardson1997}
Thomas~S. Richardson.
\newblock A characterization of {M}arkov equivalence for directed cyclic
  graphs.
\newblock \emph{International Journal of Approximate Reasoning}, 17:\penalty0
  107--162, 1997.

\bibitem[Richardson(2003)]{richardson2003}
Thomas~S. Richardson.
\newblock Markov properties for acyclic directed mixed graphs.
\newblock \emph{Scandinavian Journal of Statistics}, 2003.

\bibitem[Richardson and Spirtes(2002)]{richardson2002}
Thomas~S. Richardson and Peter Spirtes.
\newblock Ancestral graph markov models.
\newblock \emph{The Annals of Statistics}, 30\penalty0 (4):\penalty0 962--1030,
  2002.

\bibitem[R{\o}ysland(2012)]{roysland2012}
Kjetil R{\o}ysland.
\newblock Counterfactual analyses with graphical models based on local
  independence.
\newblock \emph{Annals of Statistics}, 40\penalty0 (4):\penalty0 2162--2194,
  2012.

\bibitem[Sch{\"o}bel and Zhu(1999)]{schobel1999}
Rainer Sch{\"o}bel and Jianwei Zhu.
\newblock Stochastic volatility with an {O}rnstein-{U}hlenbeck process: {A}n
  extension.
\newblock \emph{European Finance Review}, 3:\penalty0 23--46, 1999.

\bibitem[Schweder(1970)]{schweder1970}
Tore Schweder.
\newblock Composable {M}arkov processes.
\newblock \emph{Journal of Applied Probability}, 7\penalty0 (2):\penalty0
  400--410, 1970.

\bibitem[Shimokawa et~al.(2000)Shimokawa, Pakdaman, Takahata, Tanabe, and
  Sato]{shimokawa2000}
T.~Shimokawa, K.~Pakdaman, T.~Takahata, S.~Tanabe, and S.~Sato.
\newblock A first-passage-time analysis of the periodically forced noisy leaky
  integrate-and-fire model.
\newblock \emph{Biological Cybernetics}, 83:\penalty0 327--340, 2000.

\bibitem[Sipser(2013)]{Sipser2013}
Michael Sipser.
\newblock \emph{Introduction to the theory of computation}.
\newblock Boston: Thomson Course Technology, 3rd edition, 2013.

\bibitem[Sokol and Hansen(2014)]{sokol2014}
Alexander Sokol and Niels~Richard Hansen.
\newblock Causal interpretation of stochastic differential equations.
\newblock \emph{Electronic Journal of Probability}, 19\penalty0 (100):\penalty0
  1--24, 2014.

\bibitem[Sonntag and Pe{\~n}a(2015)]{sonntag2015}
Dag Sonntag and Jose~M. Pe{\~n}a.
\newblock Chain graph interpretation and their relations revisited.
\newblock \emph{International Journal of Approximate Reasoning}, 58:\penalty0
  39--56, 2015.

\bibitem[Spirtes et~al.(1998)Spirtes, Richardson, Meek, Scheines, and
  Glymour]{spirtes1998}
Peter Spirtes, Thomas Richardson, Christopher Meek, Richard Scheines, and Clark
  Glymour.
\newblock Using path diagrams as a structural equation modeling tool.
\newblock \emph{Sociological Methods and Research}, 27\penalty0 (2):\penalty0
  182--225, 1998.

\bibitem[Stein and Stein(1991)]{stein1991}
Elias~M. Stein and Jeremy~C. Stein.
\newblock Stock price distributions with stochastic volatility: {A}n analytic
  approach.
\newblock \emph{The Review of Financial Studies}, 4\penalty0 (4):\penalty0
  727--752, 1991.

\bibitem[Verma and Pearl(1991)]{vermaEquiAndSynthesis}
Thomas Verma and Judea Pearl.
\newblock Equivalence and synthesis of causal models.
\newblock Technical Report {R}-150, University of {C}alifornia, {L}os
  {A}ngeles, 1991.

\end{thebibliography}

\end{document}